\documentclass[reqno]{amsart}

\usepackage{macros}

\begin{document}
\title{Chevalley Polytopes and Newton-Okounkov bodies}
\author{Peter Spacek}
\address{Technische Universit\"at Chemnitz, Germany}
\email{peter.spacek@math.tu-chemnitz.de}

\author{Charles Wang}
\address{University of Michigan, USA}
\email{cmwa@umich.edu}
\date{\today}

\begin{abstract}
We construct a family of polytopes, which we call \emph{Chevalley polytopes}, associated to homogeneous spaces $\X=\G/\P$ in their projective embeddings $\X\hookrightarrow \bP(V_{\varpi})$ together with a choice of reduced expression for the minimal coset representative $\wP$ of $\wo$ in $\weyl/\weylp$. When $\X$ is minuscule in its minimal embedding, we describe our construction in terms of order polytopes of minuscule posets and use the associated combinatorics to show that minuscule Chevalley polytopes are Newton-Okounkov bodies for $\X$ and that the Pl\"ucker coordinates on $\X$ form a Khovanskii basis for $\C[\X]$. We conjecture similar properties for general $\X$ and general embeddings $\X\hookrightarrow\bP(V_\varpi)$, along with a remarkable decomposition property which we consider as a polytopal shadow of the Littlewood-Richardson rule. We highlight a connection between Chevalley polytopes and string polytopes and give examples where Chevalley polytopes possess better combinatorial properties than string polytopes. We conclude with several examples further illustrating and supporting our conjectures.
\end{abstract}

\maketitle
\vspace{-.5cm}
\setcounter{tocdepth}{2}
\tableofcontents

\section{Introduction}
The geometry of homogeneous spaces has many fruitful connections with polyhedral geometry, and in turn these polytopes provide powerful combinatorial techniques for studying the geometry of homogeneous spaces. To give a few examples, Mirkovi\'c-Vilonen cycles give rise to MV polytopes as their images under the moment map \cite{anderson_mv_polytope}, string polytopes parametrize crystal bases for irreducible representations of $\G/\P$ and furthermore are \emph{Newton-Okounkov bodies} \cite{littelman_cones_crystals_patterns}, the lattice points of Gel'fand-Tsetlin polytopes parametrize distinguished bases for irreducible representations of $\G$ \cite{gelfand_tsetlin}. In this article, we construct a polytope, which we call a \emph{Chevalley polytope}, associated to the triple $(\X,\varpi,\bs)$ where
\begin{enumerate}
    \item $\X=\G/\P$ is a homogeneous space for the action of a simple, semi-simple, complex algebraic group $\G$ with associated parabolic subgroup $\P$,
    \item $\varpi$ denotes a dominant integral weight of $\G$ whose corresponding highest weight line in the irreducible $\G$-representation $V_\varpi$ is stabilized by $\P$ and hence induces a projective embedding $\X\hookrightarrow\bP(V_\varpi)$ into the projectivization of $V_\varpi$, and
    \item $\bs=s_{r_{\ellwP}}\cdots s_{r_1}$ is a choice of reduced expression for the minimal coset representative $\wP$ of the coset $\wo\weylp\in\weyl/\weylp$, where $\ellwP=\ell(\wP)$.
\end{enumerate}
Using this triple, we construct in Section \ref{sec:construction} a \emph{valuation} (over $\C$) $\nu_{\X,\varpi,\bs}:\C[\X]\rightarrow\Z^{\ellwP}$ on the homogeneous coordinate ring $\C[\X]$ of $\X$ in its projective embedding $\X\hookrightarrow\bP(V_\varpi)$, noting that $\dim(\X)=\ellwP$. We define $\nu_{\X,\varpi,\bs}$ by sending $p\in\C[\X]$ to the exponent vector of the degree reverse lexicographic minimal term of its restriction $p|_{\Xo}$ to an algebraic torus $\Xo=\opendunim/\P\subset\X$, whose coordinates are parametrized by the choice of $\bs$. We then define the Chevalley polytope $\cP_{\X,\varpi,\bs}$ to be the convex hull of the valuations of the homogeneous degree-$1$ part of $\C[\X]$:
\[
\cP_{\X,\varpi,\bs} = \conv\left(\{\nu_{\X,\varpi,\bs}(p)\mid p\in\C[\X],~\deg(p)=1\}\right).
\] 
Such valuations appear frequently in the context of \emph{Khovanskii bases} and \emph{Newton-Okounkov bodies}, and our valuation is closely related to, but slightly different from, the \emph{string parametrization} originally studied in \cite{littelman_cones_crystals_patterns, BZ-tensor-product-mult} and later in \cite{kaveh_strings}. We remark upon the relationship between Chevalley polytopes and string polytopes in Section \ref{sec:string} and we give examples of Chevalley polytopes that possess better combinatorial properties than the corresponding string polytopes. 

In Section \ref{sec:minuscule}, we study Chevalley polytopes arising from \emph{minuscule} spaces $\X=\G/\P$ (where $\P=\P_k$ is a maximal parabolic subgroup stabilizing the highest weight line of the minuscule fundamental representation $V_{\fwt[k]}$) in their minimal embeddings $\X\hookrightarrow\bP(V_{\fwt[k]})$, also known as their \emph{(generalized) Pl\"ucker embeddings}. Due to the \emph{full-commutativity} of $\wP$ in this case \cite[Theorem 6.1]{Stembridge_Fully_commutative_elements_of_Coxeter_groups}, both the valuation and Chevalley polytope do not depend on the reduced expression $\bs$ (up to permutations of the coordinates of $\Z^{\ellwP}$) yielding an intrinsic valuation $\nu_{\X,\fwt[k]}$ and polytope $\cP_{\X,\fwt[k]}$.
We furthermore give a combinatorial interpretation of the valuations of Pl\"ucker coordinates for $\X$ in terms of filters of the associated \emph{minuscule poset} $\minposet$ and we use this to show that the Chevalley polytope $\cP_{\X,\fwt[k]}$ is exactly the \emph{order polytope} $\cO_{\minposet}$ associated to $\minposet$. Then, using the combinatorial properties of order polytopes, we establish the following theorem:

\begin{thm}[Theorems \ref{thm:nobody} and \ref{thm:plucker-khovanskii} below]
Let $\X=\G/\P_k\hookrightarrow\bP(V_{\fwt[k]})$ be a minuscule homogeneous space in its minimal embedding. The Chevalley polytope $\cP_{\X,\fwt[k]}$ is a Newton-Okounkov body for $\X$ with respect to the valuation $\nu_{\X,\fwt[k]}$. Furthermore, the Pl\"ucker coordinates on $\X$ form a Khovanskii basis for $\C[\X]$ with respect to the valuation $\nu_{\X,\fwt[k]}$. Finally, there is a toric degeneration of $\X$ to the projectively normal toric variety associated to $\cP_{\X,\fwt[k]}$. 
\end{thm}

While Khovanskii bases and Newton-Okounkov bodies are generally difficult to construct, these theorems give straightforward constructions for both when $\X$ is minuscule. Outside of the minuscule case, the situation is less clear. In Section \ref{sec:general}, we generalize the combinatorial interpretation of the valuations of Pl\"ucker coordinates studied in the minuscule case to general $\X$ and general embeddings $\X\hookrightarrow\bP(V_\varpi)$ by defining \emph{weighted embeddings} of \emph{weighted heaps} into the heap $\sH_{\bs}$ associated to $\bs$. A \emph{heap} is a poset associated to an expression $\bs$ of a Weyl group element and is invariant under commutations in $\bs$; in particular when $\X$ is minuscule, the heap $\sH_{\bs}$ coincides with $\minposet$ for any choice of $\bs$ and we thus view heaps as direct generalizations of minuscule posets. The difficulty in the more general setting is twofold: first, the valuation and polytope both depend heavily upon the choice of $\bs$ and furthermore we need to consider weighted embeddings of weighted heaps which do not arise as filters of the heap $\sH_{\bs}$ associated to $\bs$.

Furthermore, compared to the minuscule case, we replace Pl\"ucker coordinates with a set of coordinates corresponding to a weight basis derived from the actions of monomials in the Chevalley generators of the Lie algebra $\g$ associated to $\G$, which we refer to as a \emph{Chevalley basis for $V_\varpi$}. This reason for this replacement is that minuscule representations admit a canonical choice of weight basis up to (global) scaling yielding a canonical choice of Pl\"ucker coordinates on $\X$, while general representations do not admit such a canonical choice. In this more general setting, the set of coordinates coming from a Chevalley basis for $V_\varpi$ provides a natural generalization of the set of Pl\"ucker coordinates on $\X$. Generalizing the statements we obtained in the minuscule setting, we present the following conjectures for arbitrary $\X\hookrightarrow\bP(V_\varpi)$ and $\bs$:

\begin{conj}[Conjectures \ref{conj:nobody} and \ref{conj:khovanskii} below]
Let $\X=\G/\P\hookrightarrow \bP(V_{\varpi})$ be a homogeneous space with a projective embedding, then there exists a reduced expression $\bs$ for $\wP$ such that the Chevalley polytope $\cP_{\X,\varpi,\bs}$ is a Newton-Okounkov body for $\X$ with respect to the valuation $\nu_{\X,\varpi,\bs}$. For any such $\bs$, the coordinates on $\X$ corresponding to any \emph{Chevalley basis for $V_\varpi$} form a Khovanskii basis for $\C[\X]$ with respect to the valuation $\nu_{\X,\varpi,\bs}$. 
\end{conj}

In general, we expect there to be numerous reduced expressions $\bs$ for which $\cP_{\X,\varpi,\bs}$ is a Newton-Okounkov body for $\X$ and for which the coordinates corresponding to any Chevalley basis for $V_\varpi$ form a Khovanskii basis for $\C[\X]$, both with respect to the valuation $\nu_{\X,\varpi,\bs}$. However, there are also reduced expressions $\bs$ for which we require higher degree elements of $\C[\X]$ to form a Khovanskii basis and whose rescaled valuations are necessary to obtain a Newton-Okounkov body. In Appendix \ref{sec:examples}, we present some examples where all reduced expressions satisfy the conjecture above, as well as some examples where the conjecture only holds for certain reduced expressions. Finally, we also observed the following remarkable decomposition property:

\begin{conj}[Conjecture \ref{conj:decomposition} below]
Let $\X=\G/\P\hookrightarrow\bP(V_\varpi)$ be a homogeneous space with a projective embedding, and let $\varpi=\sum_i c_i\fwt[i]$ be a dominant integral weight for $\G$ with its fundamental weight basis expansion. Then for any reduced expression $\bs$ for $\wP$, there exist reduced expressions $\bs_j$ for the minimal coset representatives $w^{P_j}$ such that 
\[
\cP_{\X,\varpi,\bs} = \sum\nolimits_j c_j\cdot \cP_{\G/\P_j, \fwt[j], \bs_j},
\]
where the sum denotes the Minkowski sum and where each $\cP_{\G/\P_j,\fwt[j],\bs_j}$ is embedded into by identifying $w^{P_i}$ as a subword of $\wP$. 
\end{conj}
This conjecture is reminscent of the Littlewood-Richardson rule in that it identifies the irreducible representation $V_\varpi$ as the ``leading term'' in the decomposition of the tensor product $\bigotimes_j V_{\fwt[j]}^{\otimes c_j}$ into irreducible representations. (The other irreducible representations in the decomposition of the tensor product come from polytopes defined by subsets of the vertices of $\cP_{\X,\varpi,\bs}$.)

We conclude with a brief outline. In Section \ref{sec:prelims}, we fix notations and conventions, and we recall relevant Lie-theoretic and order-theoretic background as well as Khovanskii bases and Newton-Okounkov bodies. We then define our valuations and Chevalley polytopes in Section \ref{sec:construction}. In Section \ref{sec:minuscule}, we study Chevalley polytopes in the case when $\X\hookrightarrow\bP(V_{\fwt[k]})$ is minuscule in its minimal embedding and show that the minuscule Chevalley polytopes are Newton-Okounkov bodies for $\X$ and also that the Pl\"ucker coordinates form a Khovanksii basis for $\C[\X]$. We then conjecture analogous results for the general case in Section \ref{sec:general}, which we plan to investigate in future work. In Section \ref{sec:string}, we compare our construction with string polytopes and give several examples where our polytopes possess better combinatorial properties. We conclude by presenting evidence for our conjectures with a number of further examples in Appendix \ref{sec:examples}.
\section{Background and Conventions} 
\label{sec:prelims}
In this section, we will fix our conventions and notation and also summarize necessary background and results. Unless otherwise stated, we will always use the notation and conventions fixed below. We begin with Lie-theoretic background in Section \ref{sec:conventions}, followed by combinatorial details of minuscule representations and the corresponding minuscule posets in \ref{sec:minuscule_background}. We briefly introduce order polytopes in Section \ref{sec:order_polytopes} and derive some basic corollaries we will need to study minuscule Chevalley polytopes. Finally, we introduce Khovanskii bases and Newton-Okounkov bodies in \ref{sec:khovanskii+no}.

\subsection{Conventions and notation} 
\label{sec:conventions}
Let $\G$ be a simple and simply-connected complex algebraic group of rank $n$ and fix a maximal torus $\torus$ and opposite Borel subgroups $\borelp,\borelm\supset\torus$, yielding decompositions $\borelp=\torus\unip$ and $\borelm=\torus\unim$ for opposite unipotent subgroups $\unip,\unim$. The parabolic subgroups $\P$ of $\G$ (those subgroups which contain $\borelp$) are in one-to-one correspondence with subsets $I\subset[n]$, where we write $[n]=\{1,\dots,n\}$, and the parabolic Weyl subgroup $\weylp$ of the Weyl group $\weyl$ of $\G$ is generated by all simple reflections $s_i$ with $i\notin I$. We write $\wo$ for the longest element of $\weyl$ and $\wP$ for the minimal coset representative of $\wo\weylp\in\weyl/\weylp$; its length is denoted $\ellwP=\ell(\wP)$. We write $\dunimP$ for the unipotent cell $\dunimP = \unim\cap\borelp\wP\borelp \subset\G$ (for some choice of representative in $\G$ of $\wP$). Note that in the previous works \cite{Spacek_Wang_exceptional, Spacek_Wang_cominuscule} we considered unipotent cells $\dunimP$ inside of the Langlands dual group $\dG$, while in this paper we will instead consider them inside of $\G$; for the same reason, we take $\dunimP$ to be the cell corresponding to $\wP$ instead of the cell corresponding to $(\wP)^{-1}$.

Let $\roots$ and $ \droots$ respectively denote the roots and coroots of $\G$ with respect to $\torus$, and let $\{ \sr_1, \ldots, \sr_n\} \subset\roots$ denote the base of simple roots for $\G$ fixed by the choice of $\borelp$ with corresponding basis of fundamental weights $\{\fwt[1],\dots, \fwt[n]\}$. Writing $\g$ and $\cartan$ for the Lie algebras of $\G$ and $\torus$ respectively, we find the root space decomposition $\g=\cartan\op\bigoplus_{\sr\in\roots}\g_\sr$, and we fix Chevalley generators $(\Che_i,\Chh_i,\Chf_i)_{i\in[n]}$ for $\g$, where $\Che_i\in\g_{\sr_i}$, $\Chf_i\in\g_{-\sr_i}$ and $\Chh_i=[\Che_i,\Chf_i]\in\cartan$. The Chevalley generators define one-parameter subgroups
\[
\x_i(a) = \exp(a\,\Che_i)\in\unip \qand \y_i(a) = \exp(a\Chf_i)\in\unim \qfor a\in\C.
\]
Since we assume $\G$ is simply-connected, all fundamental representations of $\g$ give rise to fundamental representations of $\G$, and we will frequently (and implicitly) work with representations of $\g$ and $\G$ interchangeably. For each reduced expression $\bs= s_{r_{\ellwP}}\cdots s_1$ for $\wP$, we define the open, dense algebraic torus
\begin{equation}
\opendunim = \{y_{r_{\ellwP}}(a_{\ellwP})\cdots y_{r_1}(a_1)~|~a_i\in\C^*\}~~\subset~~\dunimP,    
\label{eq:df_opendunim_red_exp}
\end{equation}
which is parametrized by the \emph{toric coordinates} $a_1,\dots, a_{\ellwP}\in\C^*$. In general, the torus $\opendunim$ depends on the particular choice of reduced expression $\bs$. However, in the case that two adjacent simple reflections commute $s_{r_{i+1}}s_{r_i} = s_{r_i}s_{r_{i+1}}$, the same holds for the corresponding Chevalley generators $\Chf_{r_{i+1}}\Chf_{r_i} = \Chf_{r_i}\Chf_{r_{i+1}}$, so we find that $\y_{r_{i+1}}(a_{i+1})\y_{r_i}(a_i) = \y_{r_i}(a_i)\y_{r_{i+1}}(a_{i+1})$, and hence that $\opendunim$ is invariant under commutations in $\bs$ up to a permutation of its coordinates. Because of this invariance, we note that $\opendunim$ in fact only depends on the \emph{heap} $\sH_{\bs}$ associated to $\bs$, which is the poset with elements $\pb_{\ellwP},\dots,\pb_1$ and partial ordering $\pb_i > \pb_j$ if $i<j$ and $s_is_j\neq s_js_i$. Note that by definition of $\sH_{\bs}$, there is a natural correspondence between the elements $\pb_i\in\sH_{\bs}$ and simple reflections $s_{r_i}$ in $\bs$ and we label each $\pb\in\sH_{\bs}$ with the index of the corresponding simple reflection ${r_i}$, written as $\indx(\pb)=r_i$. Thus, we can rewrite \eqref{eq:df_opendunim_red_exp} as
\[
\opendunim = \left\{\left.\tprod[_{\pb\in\sH_{\bs}}] y_{\indx(\pb)}(a_\pb)~\right|~a_\pb\in\C^*\right\},
\]
where the product is ordered left-to-right in increasing order (i.e. with the largest element on the right) using any linear extension of the partial order on $\sH_{\bs}$. 

We will consider homogeneous spaces $\X=\G/\P$ as projective varieties using the embeddings $\iota:\X\hookrightarrow\bP(V_\varpi)$, where $\varpi$ is a dominant integral weight of $\G$ whose corresponding highest weight line in the irreducible $\G$-representation $V_\varpi$ with highest weight $\varpi$ is stabilized by $\P$. Explicitly, the embedding of $\X=\G/\P$ is given by the action on the highest weight line $g\P\mapsto g\P\cdot\C v_\varpi$ and yields a closed embedding of $\X$ into $\bP(V_\varpi)$. Letting $I\subset[n]$ denote the subset of indices corresponding to $\P$, the special case when $\varpi=\sum_{i\in I}\fwt$ is the minimal embedding of $\X$, sometimes called the \emph{(generalized) Pl\"ucker embedding}, especially when $|I|=1$. 

These embeddings are all projectively normal, which will be a useful technical condition in studying toric degenerations of these spaces. To see this, we note that the global sections of $\cO(i)$ on $\bP(V_\varpi)$ can be identified with $\Sym^i(V_\varpi)$ while the global sections of the pullback $\iota^*(\cO(i))$ to $\X$ can be identified with $V_{i\varpi}$ by the Borel-Weil-Bott theorem. The latter is an irreducible $\G$-representation and the corresponding restriction morphism $\Sym^i(V_\varpi)\rightarrow V_{i\varpi}$ is nonzero, so must be surjective for any $i$.

Finally, any weight basis $\cB$ for $V_\varpi$ induces a set of coordinates on $\X$ which form a set of algebra generators for the homogeneous coordinate ring $\C[\X]$ as well as a basis of the homogeneous degree-$1$ part of $\C[\X]$. To describe the coordinates coming from a weight basis explicitly, let $\cB^*$ denote the dual basis to $\cB$ and $v_\mu\in\cB$ be a weight vector; then we obtain the corresponding coordinate function $p_{v_\mu}$ via the projection formula
\begin{equation}
\label{eq:coordinate}
    p_{v_\mu}(gP) = v_\mu^*(gP\cdot\C v_\varpi).
\end{equation}
While this construction is useful for producing a set of algebra generators of $\C[\X]$, the choice of a weight basis $\cB$ is generally far from canonical unless $V_\varpi$ is minuscule. There are various bases for such representations in the literature, and we will be interested in bases arising as follows. Let $\boldf$ denote a monomial $\Chf_{i_\ell}^{d_\ell}\cdots \Chf_{i_1}^{d_1}$ in the Chevalley generators $\Chf_i$. Then we call any basis $\cB$ for $V_\varpi$ obtained as a subset of $\{\boldf\cdot v_\varpi~|~\boldf\cdot v_\varpi\neq0\}$ a \emph{Chevalley basis for $V_\varpi$}. 

\subsection{Minuscule representations and minuscule posets}
\label{sec:minuscule_background}
A fundamental weight $\fwt[k]$ is \emph{minuscule} if $\llan\fwt[k],\sdr\rran\in\{-1,0,+1\}$ for any coroot $\sdr\in\droots$, where $\llan\cdot,\cdot\rran$ denotes the dual pairing between the character and cocharacter lattices of $\torus$. When $\P=\P_k$ is the parabolic subgroup stabilizing the highest weight line of $V_{\fwt[k]}$, the homogeneous space $\X=\G/\P$ is often called \emph{minuscule}. In Table \ref{tab:CominusculeSpaces}, we have listed all minuscule fundamental weights together with the associated homogeneous spaces.

\begin{table}[b!h]%
\[
\begin{array}{ccc|c}
\multicolumn{3}{c|}{\text{type of $\G$ and weight $\fwt[k]$}} & \text{variety $\G/\Pj[i]$} \vphantom{\dfrac MM}\\\hline
A_{n-1} & 
\begin{tikzpicture}[baseline = -0.75ex, scale = 0.625]
\coordinate (1) at (0,0);
\coordinate (2) at (0.625,0);
\coordinate (3) at (1.25,0);
\coordinate (3+) at (1.75,0);
\coordinate (n-) at (2.25,0);
\coordinate (n) at (2.75,0);
\draw[thick] (1)--(3+);
\draw[thick, dotted] (3+)--(n-);
\draw[thick] (n-)--(n);
\draw[black,fill=black] (1) circle (3pt); 
\draw[black,fill=black] (2) circle (3pt);
\draw[black,fill=black] (3) circle (3pt);
\draw[black,fill=black] (n) circle (3pt);
\end{tikzpicture}
&\text{any }k & \Gr(k,n) \vphantom{\dfrac MM}\\
B_n & 
\begin{tikzpicture}[baseline = -0.75ex, scale = 0.625]
\coordinate (1) at (0,0);
\coordinate (2) at (0.625,0);
\coordinate (2+) at (1.125,0);
\coordinate (n-1-) at (1.625,0);
\coordinate (n-1) at (2.125,0);
\coordinate (n) at (3.125,0);
\draw[thick] (1)--(2+);
\draw[thick, dotted] (2+)--(n-1-);
\draw[thick] (n-1-)--(n-);
\draw[thick,double,double distance = 1pt] (n-1)--(n);
\draw[black,fill=white] (1) circle (3pt); 
\draw[black,fill=white] (2) circle (3pt);
\draw[black,fill=white] (n-1) circle (3pt);
\draw[black,fill=black] (n) circle (3pt);
\node at (2.625,0) {$\boldsymbol{>}$};
\end{tikzpicture}
& n & \OG(n,2n+1) \vphantom{\dfrac MM}\\
C_n & 
\begin{tikzpicture}[baseline = -0.75ex, scale = 0.625]
\coordinate (1) at (0,0);
\coordinate (2) at (0.625,0);
\coordinate (2+) at (1.125,0);
\coordinate (n-1-) at (1.625,0);
\coordinate (n-1) at (2.125,0);
\coordinate (n) at (3.125,0);
\draw[thick] (1)--(2+);
\draw[thick, dotted] (2+)--(n-1-);
\draw[thick] (n-1-)--(n-);
\draw[thick,double,double distance = 1pt] (n-1)--(n);
\draw[black,fill=black] (1) circle (3pt); 
\draw[black,fill=white] (2) circle (3pt);
\draw[black,fill=white] (n-1) circle (3pt);
\draw[black,fill=white] (n) circle (3pt);
\node at (2.625,0) {$\boldsymbol{<}$};
\end{tikzpicture}
& 1 & \CP^{2n-1}\vphantom{\dfrac MM}\\
D_n & 
\begin{tikzpicture}[baseline = -0.75ex, scale = 0.625]
\coordinate (1) at (0,0);
\coordinate (2) at (0.625,0);
\coordinate (2+) at (1.125,0);
\coordinate (n-2-) at (1.625,0);
\coordinate (n-2) at (2.125,0);
\coordinate (n-1) at (2.75,0.2);
\coordinate (n) at (2.75,-0.2);
\draw[thick] (1)--(2)--(2+);
\draw[thick, dotted] (2+)--(n-2-);
\draw[thick] (n-2-)--(n-2)--(n-1);
\draw[thick] (n-2)--(n);
\draw[black,fill=black] (1) circle (3pt); 
\draw[black,fill=white] (2) circle (3pt);
\draw[black,fill=white] (n-2) circle (3pt);
\draw[black,fill=white] (n-1) circle (3pt);
\draw[black,fill=white] (n) circle (3pt);
\end{tikzpicture}
& 1 & Q_{2n-2}\vphantom{\dfrac MM}\\
D_n & 
\begin{tikzpicture}[baseline = -0.75ex, scale = 0.625]
\coordinate (1) at (0,0);
\coordinate (2) at (0.625,0);
\coordinate (2+) at (1.125,0);
\coordinate (n-2-) at (1.625,0);
\coordinate (n-2) at (2.125,0);
\coordinate (n-1) at (2.75,0.2);
\coordinate (n) at (2.75,-0.2);
\draw[thick] (1)--(2)--(2+);
\draw[thick, dotted] (2+)--(n-2-);
\draw[thick] (n-2-)--(n-2)--(n-1);
\draw[thick] (n-2)--(n);
\draw[black,fill=white] (1) circle (3pt); 
\draw[black,fill=white] (2) circle (3pt);
\draw[black,fill=white] (n-2) circle (3pt);
\draw[black,fill=black] (n-1) circle (3pt);
\draw[black,fill=black] (n) circle (3pt);
\end{tikzpicture}
& n-1\text{ or }n  & \OG(n,2n)\vphantom{\dfrac MM}\\
E_6 & 
\begin{tikzpicture}[baseline = 0.25ex, scale = 0.625]
\coordinate (1) at (0,0);
\coordinate (3) at (0.625,0);
\coordinate (4) at (1.25,0);
\coordinate (5) at (1.875,0);
\coordinate (6) at (2.5,0);
\coordinate (2) at (1.25,0.625);
\draw[thick] (1)--(3)--(4)--(5)--(6);
\draw[thick] (4)--(2);
\draw[black,fill=black] (1) circle (3pt); 
\draw[black,fill=white] (2) circle (3pt);
\draw[black,fill=white] (3) circle (3pt);
\draw[black,fill=white] (4) circle (3pt);
\draw[black,fill=white] (5) circle (3pt);
\draw[black,fill=black] (6) circle (3pt);
\end{tikzpicture}
& 1\text{ or }6 & \OP^2 = \LGE^\SC_6/P_6 \vphantom{\dfrac MM}\\
E_7 & 
\begin{tikzpicture}[baseline = 0.25ex, scale = 0.625]
\coordinate (1) at (0,0);
\coordinate (3) at (0.625,0);
\coordinate (4) at (1.25,0);
\coordinate (5) at (1.875,0);
\coordinate (6) at (2.5,0);
\coordinate (7) at (3.125,0);
\coordinate (2) at (1.25,0.625);
\draw[thick] (1)--(3)--(4)--(5)--(6)--(7);
\draw[thick] (4)--(2);
\draw[black,fill=white] (1) circle (3pt); 
\draw[black,fill=white] (2) circle (3pt);
\draw[black,fill=white] (3) circle (3pt);
\draw[black,fill=white] (4) circle (3pt);
\draw[black,fill=white] (5) circle (3pt);
\draw[black,fill=white] (6) circle (3pt);
\draw[black,fill=black] (7) circle (3pt);
\end{tikzpicture}
& 7 & \LGE^\SC_7/P_7 \vphantom{\dfrac MM}
\end{array}
\]
\caption{(Adapted from \cite{Spacek_Wang_cominuscule}.) All minuscule fundamental weights $\fwt[k]$ and the corresponding minuscule spaces $\G/\P_k$.}
\label{tab:CominusculeSpaces}
\end{table}

While general representations do not admit canonical choices of weight bases (even up to scaling), minuscule representations $V_{\fwt[k]}$ have the special property that all weight spaces are one-dimensional, so for minuscule representations there is a canonical choice of weight basis up to (global) scaling, and the resulting projective coordinates are often called \emph{(generalized) Pl\"ucker coordinates} on $\X$. Explicitly, $V_{\fwt[k]}$ has a highest weight line $\C v_{\fwt[k]}$ and for any weight vector $v_\mu$, there is a corresponding (generalized) Pl\"ucker coordinate $p_{v_\mu}$ defined by equation \eqref{eq:coordinate}.

Before going further, we will need some basic order-theoretic concepts. First, an \emph{order ideal} of a poset $\sP$ is a subset $\sI\subset\sP$ such that $x,y\in\sP$ with $y\in\sI$ and $x\le y$ implies that $x\in \sI$. Dually, a \emph{filter} of a poset $\sP$ is a subset $\sF\subset\sP$ such that $x,y\in\sP$ with $x\in\sF$ and $x\le y$ implies that $y\in \sF$. For any $x,y\in \sP$, the \emph{join} $x\vee y$ (if it exists) is the least upper bound of $x$ and $y$ in $\sP$ and dually the \emph{meet} $x\wedge y$ (if it exists) is the greatest lower bound of $x$ and $y$ in $\sP$. If $x\in \sP$ cannot be written as the join of elements distinct from itself, then we say that $x$ is \emph{join-irreducible}. The \emph{dual poset} $\sP^*$ of a poset $\sP$ is the poset on the elements of $\sP$ obtained by reversing the order on $\sP$. A poset $\sP$ is called a \emph{lattice} if every $x,y\in\sP$ have a meet and join and a lattice is called \emph{distributive} if the meet and join distribute over each other. Birkhoff showed in \cite{birkhoff_representation} that distributive lattices can be identified with the set of order ideals of their posets of join-irreducible elements. 

Returning to minuscule representations, Proctor showed in \cite{proctor_bruhat_lattices} that the poset of weights of a minuscule representation $V_{\fwt[k]}$ using the Bruhat order (i.e.~$\mu_1>\mu_2$ if and only if $\mu_1-\mu_2$ is a positive linear combination of simple roots) is a distributive lattice. The associated poset of join-irreducible elements of the weight lattice of a minuscule representation is often called a \emph{minuscule poset}, which we denote by $\minposet$. Furthermore, the dual representation $V_{\fwt[k]}^*$ is a minuscule representation of highest weight $\fwt[\sigma_0(k)]$, where $\sigma_0(k):[n]\rightarrow[n]$ is the permutation defined by the action of $\wo$ on fundamental weights: $\fwt[\sigma_0(i)] = -\wo\cdot\fwt$. The corresponding minuscule poset $\ideal^{\P_{\sigma_0(k)}}$ is the dual poset to $\minposet$, so thus we obtain identifications of weights of $V_{\fwt[k]}$ with order ideals of $\minposet$ or equivalently with order ideals of $\ideal^{\P_{\sigma_0(k)}} = (\minposet)^*$, which are dually filters of $\minposet$. Thus, we may equivalently consider weights of the minuscule representation $V_{\fwt[k]}$ as filters of $\minposet$. Since minuscule posets can all be drawn as skew Young diagrams, we may thus label weights of minuscule representations (and hence also Pl\"ucker coordinates of minuscule spaces) by skew Young diagrams which fit inside a certain shape. 

It is easy to see that a weight for $V_{\fwt[k]}$ is join-irreducible when there is exactly one simple reflection lowering the weight. Hence, each element $\pb\in\minposet$ corresponds to the simple reflection that lowers the corresponding weight, and we label each $\pb\in\minposet$ by its corresponding simple reflection. We will write $\indx(\pb)=i$ for the index of the simple reflection corresponding to $\pb\in\minposet$. With this labeling, we note that the heap $\sH_{\bs}$ associated to $\bs$ is identical to the order ideal $\minposet$ for any reduced expression $\bs$ for $\wP$ due to a result of  Stembridge in \cite[Theorem 6.1]{Stembridge_Fully_commutative_elements_of_Coxeter_groups} which shows that all minimal coset representative of the parabolic quotient $\weyl/\weylp$ are \emph{fully commutative} when $\weylp$ stabilizes a minuscule weight.

\begin{ex}\label{ex:list_of_minposets}
We listed all the minuscule weights in Table \ref{tab:CominusculeSpaces}, and here we provide representative examples of the corresponding minuscule posets (adapted from \cite[Example 2.18]{Spacek_Wang_cominuscule}).\begin{equation*}
\label{eq:young_diagrams}
\Yboxdim{7pt}
\begin{array}{c}
\begin{array}{c||c|c|c|c|c|c|c}
\raisebox{40pt}{$\minposet$}&
\raisebox{40.5pt}{\scriptsize\young(3456,2345,1234)} &
\raisebox{27pt}{\scriptsize\young(5,45,345,2345,12345)}& 
\raisebox{27pt}{\scriptsize\young(12345,::::4,::::3,::::2,::::1)}& 
\raisebox{27pt}{\scriptsize\young(12345,:::64,::::3,::::2,::::1)}& 
\raisebox{27pt}{\scriptsize\young(6,45,346,2345,12346)}& 
\raisebox{33.5pt}{\scriptsize\young(65431,::243,:::542,:::65431)}& 
\raisebox{0pt}{\scriptsize\young(765431,:::243,::::542,::::65431,::::76543,:::::::24,::::::::5,::::::::6,::::::::7)}
\vspace{-2.8em}\\
\text{type} & A_6 & B_5 & C_5 & D_6 & D_6 & E_6 & E_7 \\
k           & 3   & 5   & 1   & 1   & 6   & 6   & 7
\end{array}
\end{array}
\end{equation*}
Note that choosing $k=n-1$ in type $\LGD_n$ switches the labels $n-1$ and $n$ along the main diagonal of the staircase partition, and that choosing $k=1$ in type $\LGE_6$ yields the dual poset.
\end{ex}

The correspondence between weights of $V_{\fwt[k]}$ and filters of $\minposet$ gives us another way to index the natural weight basis of $V_{\fwt[k]}$, namely $\{v_\sF~|~\sF\subset\minposet~\textrm{is a filter}\}$. Under this labeling, we note that the highest weight vector is labeled by the empty filter $\varnothing$. We thus label (generalized) Pl\"ucker coordinates similarly as
\[
p_{\sF}(gP) = v_{\sF}^*(gP\cdot\C v_\varnothing).
\]

\subsection{Order Polytopes}
\label{sec:order_polytopes}
In \cite{stanley_poset_poly}, Stanley defined two natural polytopes associated to any poset, the \emph{order polytope} and the \emph{chain polytope}. In this article, we will work solely with the order polytope, though the chain polytope is equally important in its own right.

\begin{df}[{\cite[Definition 1.1]{stanley_poset_poly}}]
    For a poset $\sP$, the \emph{order polytope} $\cO_{\sP}\subset \R^{\sP}$ (with coordinates indexed by elements of $\sP$) is defined as the set $\{(p_x)_{x\in \sP}\in[0,1]^{\sP}\}$ intersected with the half spaces
    \begin{align*}
        \{ p_x\le p_y\} \quad\text{for each pair}\quad x,y\in \sP\quad\text{with}\quad x\le y.
    \end{align*}
\end{df}
In the following, we will frequently identify points $p\in\cO_{\sP}$ of order polytopes with the corresponding functions $p:\sP\rightarrow\R$ given by $p(x) = p_x$. 

Many properties of order polytopes can be derived from the corresponding posets. In particular, we will need the following characterizations of the vertices of order polytopes. Recall that a \emph{filter} $\sF$ of a poset is an upward-closed set, i.e. if $x\in\sF$ and $x\le y$, then $y\in \sF$ as well. 

\begin{prop}[{\cite[Corollary 1.3]{stanley_poset_poly}}]
\label{prop:order-vertices}
For a poset $\sP$, the vertices of the order polytope $\cO_{\sP}$ are exactly given by the indicator vectors of filters of $\sP$.
\end{prop}
Where we recall that indicator vectors are the vectors of the values of indicator functions.

More generally, Stanley describes the \emph{Ehrhart polynomials} of order polytopes, where we recall that for $m\in\Z_{>0}$, the Ehrhart polynomial of a polytope $\cO_\sP$ evaluated at $m$ is defined to be the number of lattice points of the dilation $m\cO_{\sP}$.

\begin{prop}[{\cite[Theorem 4.1]{stanley_poset_poly}}]
\label{prop:order-ehrhart}
The number of lattice points in $m\cO_{\sP}$ for $m\in\Z_{>0}$ is given by the number of poset morphisms $\eta:\sP\rightarrow[m+1]$, where $[m+1]$ is considered as a poset via the standard ordering on integers.
\end{prop}

Using this theorem, Stanley computes the volumes of order polytopes.

\begin{prop}[{\cite[Corollary 4.2]{stanley_poset_poly}}]
\label{prop:stanley_volume}
For a poset $\sP$, the (normalized) volume of $\cO_\sP$ is equal to the number of linear extensions of $\sP$, where normalized volume is defined such that a standard $d$-simplex has normalized volume $\frac{1}{d!}$.
\end{prop}

We obtain another corollary that order polytopes have no interior lattice points.

\begin{cor}
    \label{cor:no_lattice_pts}
    The order polytope $\cO_{\sP}$ has no lattice points other than its vertices. 
\end{cor}

\begin{proof}
    By Proposition \ref{prop:order-vertices}, the vertices of $\cO_{\sP}$ are exactly given by indicator vectors (and hence indicator functions) of filters of $\sP$. Furthermore, there is a natural bijection mapping a filter $\sF\subset\sP$ to the morphism $\eta_{\sF}:\sP\rightarrow\{1,2\}$ defined by setting $\eta_{\sF} = \mathbf{1}_{\sF}+1$ for $\mathbf{1}_{\sF}$ the indicator function of $\sF\subset\sP$. Inversely, for a morphism $\eta:\sP\rightarrow\{1,2\}$ we define the filter $\sF_{\eta}=\mathbf{1}_{\eta}^{-1}(1)$ via the indicator function $\mathbf{1}_{\eta} = \eta-1$. It is straightforward to check that these maps are well-defined and mutually inverse and hence bijections, so that the number of morphisms $\sP\rightarrow[2]$ is equal to the number of vertices of $\cO_{\sP}$. 

    By Proposition \ref{prop:order-ehrhart}, the number of morphisms $\sP\rightarrow[2]$ is equal to the number of lattice points of (the trivial dilation of) $\cO_{\sP}$, so we conclude that the number of lattice points of $\cO_{\sP}$ is equal to the number of vertices of $\cO_{\sP}$, so $\cO_{\sP}$ has no lattice points other than its vertices.
\end{proof}

We obtain one final corollary of (the proof of) Proposition \ref{prop:order-ehrhart} that order polytopes satisfy the \emph{integer decomposition property (IDP)}, meaning that any lattice point in the m$th$ dilation of an order polytope $\cO_{\sP}$ can be written as a sum of $m$ lattice points of $\cO_{\sP}$. Note that since order polytopes have no lattice points other than vertices, we in fact show that every lattice point in $m\cO_{\sP}$ can in fact be written as a sum of $m$ vertices of $\cO_{\sP}$. 

\begin{cor}
    \label{cor:idp}
    The order polytope $\cO_{\sP}$ satisfies the integer decomposition property. 
\end{cor}
\begin{proof}
    Let $mp\in m\cO_{\sP}$ be a lattice point of $m\cO_{\sP}$. Since $p\in\cO_{\sP}$, we have by definition that $0\le x\le 1$ for all $x\in \sP$, so that $0\le m(p_x)\le m$, and furthermore that $p_x<p_y$ for all $x<y$ in $\sP$ so that $mp$ can be identified with a poset morphism $\eta_{mp}:\sP\rightarrow\{0,\dots,m\}$. It follows immediately that (the function) $mp$ can be written as the sum of $m$ indicator functions of filters of $\sP$ via
    \[
    mp = \sum_{i=1}^m \mathbf{1}_{(\eta_{mp})^{-1}(\{i,i+1,\dots, m\})},
    \]
    showing that (the point) $mp$ can be written as a sum of $m$ (not necessarily distinct) vertices of $\cO_\sP$.
\end{proof}

Geometrically speaking, if a polytope satisfies the IDP, then the corresponding toric variety is projectively normal. 

\subsection{Khovanskii bases and Newton-Okounkov bodies} 
\label{sec:khovanskii+no}
Finally, we recall basic facts about \emph{Khovanskii bases} and \emph{Newton-Okounkov bodies} associated to \emph{valuations}. This subsection closely follows \cite[Section 5]{Spacek_Wang_exceptional}, so we only give a brief exposition. We denote by $A$ a domain that is also a finitely generated algebra over an algebraically closed field $k$ and by $\Gamma\subset \mathbb{Q}^r$ a subgroup with an ordering $\succ$ of its elements compatible with addition on $\Gamma$. 

\begin{df}[\cite{Kaveh_Manon_Khovanskii_bases}, Definition 2.1]
  A function $\nu:A\setminus\{0\}\rightarrow \Gamma$ is a \emph{valuation} (over $k$) if
  \begin{enumerate}
  \item for all $f,g\in A\setminus\{0\}$ with $f+g\neq0$, we have $\nu(f+g)\succeq \min_{\succ}(\nu(f),\nu(g))$,
  \item for all $f,g\in A\setminus\{0\}$, we have $\nu(fg)=\nu(f)+\nu(g)$, and
  \item for all $f\in A\setminus\{0\}$ and $0\neq c\in k$, we have $\nu(cf)=\nu(f)$.
  \end{enumerate}
  For $\ga\in\Gamma$, let $F_{\succeq\ga}=\{f\in A\setminus \{0\}\mid \nu(f)\succeq\ga\}\cup \{0\}$, and analogously for $F_{\succ\ga}$. The \emph{associated graded algebra} of $A$ with respect to the valuation $\nu$ is 
  \[
  \mathrm{gr}_\nu(A)=\bigoplus_{\ga\in\Gamma} F_{\succeq\ga}/F_{\succ\ga}.
  \]
  If for all $\ga\in\Gamma$, $\dim_k(F_{\succeq\ga}/F_{\succ\ga})\le 1$, then $\nu$ has \emph{one-dimensional leaves}. 
\end{df}

The one-dimensional leaves property is a useful combinatorial condition, and the following lemma provides a large family of valuations with this property. 

\begin{lem}[\cite{Spacek_Wang_exceptional}, Lemma 5.2]
  \label{lem:one-dim-leaves}
  Let $S\subset k[x_1,\dots, x_n]$ be a finitely generated subalgebra, and let $\nu$ be the valuation sending $p\in S$ to the exponent vector of its minimal term with respect to a term ordering on $k[x_1,\dots, x_n]$, then $\nu$ has one-dimensional leaves. 
\end{lem}

Although the lemma above was originally stated for the degree-lexicographic order, the proof applies equally well to any term ordering on $k[x_1,\dots, x_n]$ and we have stated the lemma as such. We now recall \emph{Khovanskii bases} associated to valuations, which are heavily reminiscent of \emph{Gr\"obner basis} associated to term orderings on polynomial rings. 

\begin{df}[\cite{Kaveh_Manon_Khovanskii_bases}, Definition 2.5]
  A (possibly infinite) subset $B\subset A$ is a \emph{Khovanskii basis} for $A$ with respect to the valuation $\nu$ if $\mathrm{gr}_{\nu}(B)$ is a set of algebra generators for $\mathrm{gr}_\nu(A)$. 
\end{df}

Khovanskii bases are generally worse-behaved than Gr\"obner bases, but we can study (some of) their properties through \emph{Newton-Okounkov bodies} associated to valuations in the case when $A$ has a $\Z_{\ge0}$-grading, which is sometimes called a \emph{positive grading}.

\begin{df}[\cite{Kaveh_Manon_Khovanskii_bases}, Definition 2.21] \label{def:no-body}
  Let $A$ be $\mathbb{Z}_{\ge0}$-graded. The \emph{Newton-Okounkov body} $\Delta(A,\nu)$ associated to a valuation $\nu$ is the closed convex set
  \[
  \Delta(A,\nu)=\overline{\mathrm{conv}\!\left(\bigcup_{i>0}\tfrac1i{\nu(A_i\setminus\{0\})}\right)},
  \]
  where $A_i$ denotes the set of homogeneous elements of $A$ of degree $i$. 
\end{df}

Although Newton-Okounkov bodies are generally far from polyhedral, the following proposition gives an important family of Newton-Okounkov bodies which are polytopal.

\begin{prop}[\cite{Kaveh_Manon_Khovanskii_bases}, Corollary 2.24, Remark 2.25]
  \label{prop:no-volume}
  Let $Y\subset\mathbb{P}^N$ be a projective variety, let $A=\C[Y]$ be the homogeneous coordinate ring of $Y$, and let $\nu$ be a valuation on $A$ with one-dimensional leaves. Then $\deg(Y)=\mathrm{Vol}(\Delta(A,\nu))$, where $\mathrm{Vol}$ refers to the normalized lattice volume (i.e.~such that a standard lattice $d$-simplex has volume $\frac{1}{d!}$). 
  
  Suppose furthermore that $A$ has a finite Khovanskii basis with respect to $\nu$. Then $\Delta(A,\nu)$ is a rational polytope and $Y$ admits a degeneration to a toric variety whose normalization is the toric variety associated to $\Delta(A,\nu)$.
\end{prop}

Under further assumptions, e.g. as in \cite[Proposition 17.4]{Rietsch_Williams_NO_bodies_cluster_duality_and_mirror_symmetry_for_Grassmannians}, we can obtain not only a toric degeneration to a toric variety whose normalization is the toric variety associated to the Newton-Okounkov body, but actually a toric degeneration to the toric variety of the Newton-Okounkov body itself, and we will be most interested in this particular case.
\section{Chevalley polytopes}
\label{sec:construction}

In this section we construct the valuation $\nu_{\X,\varpi,\bs}:\C[\X]\rightarrow\Z^{\ellwP}$ and Chevalley polytope $\cP_{\X,\varpi,\bs}$ associated to a triple $(\X,\varpi,\bs)$ consisting of a homogeneous space $\X=\G/\P$, a dominant weight $\varpi$ yielding a projective embedding $\X\hookrightarrow\bP(V_\varpi)$, and a choice of reduced expression $\bs$ for the minimal coset representative $\wP$ of the coset $\wo\weylp\in\weyl/\weylp$, where $\ellwP=\ell(\wP)=\dim(\X)$.

Recall from \ref{sec:conventions} that the heap $\sH_{\bs}$ associated to the reduced expression $\bs = s_{r_{\ellwP}}\cdots s_{r_1}$ for $\wP$ parametrizes coordinates on the torus $\Xo=\opendunim/\P\subset\X$ where
\[
\opendunim = \left\{\left.\tprod[_{\pb\in\sH_{\bs}}] \y_{\indx(\pb)}(a_{\pb})~\right|~ a_\pb\in\C^*\right\} \subset\G,
\]
and noting that $\Xo\cong\opendunim$ (see e.g. \cite[Remark 6.6]{Spacek_Wang_cominuscule} for a proof of this fact). We will use $\opendunim$ to construct a valuation $\nu_{\X,\varpi,\bs}$ by computing restrictions of $p\in \C[\X]$ to $\Xo$ and selecting the exponent vector of the degree reverse lexicographic (degrevlex) minimal term of $\C[a_\pb\mid\pb\in\sH_{\bs}]$, where the degrevlex ordering is defined with respect to the ordering of the coordinates $a_{\pb_{\ellwP}} > \cdots > a_{\pb_1}$ arising from the choice of $\bs$. We now explain how to compute the restrictions $p|_{\Xo}$. 

Since the coordinate ring $\C[\X]$ is generated as an algebra by coordinate functions as given in equation \eqref{eq:coordinate}, it suffices to compute the restrictions of coordinate functions corresponding to a weight basis of $V_\varpi$. These coordinate functions $p_{v_\mu}$ are $\P$-invariant since $\C v_\varpi$ is stabilized by $\P$, so we may equivalently consider $p_{v_\mu}$ as a function on $\G$ and compute the restrictions $p_{v_\mu}|_{\opendunim}$. 

In principle, in order to work concretely with elements $u_-\in\opendunim$, we need to consider the series expansions of each one-parameter subgroup
\[
\y_{i}(a_\pb) = \exp(a_\pb\Chf_{i}) = 1 + a_\pb\Chf_{i} + \tfrac12 a_\pb^2\Chf_{i}^2 + \ldots, \qwhere i=\indx(\pb)
\]
inside the completion of the universal enveloping algebra of the Lie algebra of $\unim$. However, since we only consider the action of $u_-$ on a fixed finite-dimensional (irreducible) representation $V_\varpi$, it suffices to truncate each one-parameter subgroup expansion to a sufficiently large, finite degree. 
Hence, each term of the expansion for $u_-\in\opendunim$ in the completed universal enveloping algebra of the Lie algebra of $\unim$ is a monomial $\Chf_{j_1}^{d_1}\cdots \Chf_{j_l}^{d_l}$ in the Chevalley generators $\Chf_i$ with coefficient a homogeneous polynomial in the torus coordinates $a_\pb$ of degree $(d_1+\cdots+d_l)$:
\begin{equation}
\label{eq:torus_expansion}
    u_- = 1 + \sum_{i} \left(\sum_{\pb} a_\pb\right)\Chf_i 
    + \sum_{i}\left(\sum_{\pb}\tfrac12a_\pb^2\right)\Chf_i^2+\sum_{j\neq i}\left(\sum_{\pb'<\pb}a_{\pb'}a_{\pb}\right)\Chf_j\Chf_i
    +\ldots
\end{equation}
where $i,j\in[n]$ and $\pb,\pb'\in\sH_{\bs}$ with $\indx(\pb)=i$ and $\indx(\pb')=j$, and
only a finite number of terms in this expansion act nontrivially on the representation.

Finally, let $p_{v_\mu}$ be a coordinate function on $\X$. It is straightforward to see that the restriction $p_{v_\mu}|_{\opendunim}$ will be a sum of the coefficients of those terms $\Chf_{j_1}^{d_1}\cdots \Chf_{j_l}^{d_l}$ of expansion \eqref{eq:torus_expansion} above such that $\Chf_{j_1}^{d_1}\cdots\Chf_{j_l}^{d_l}v_\varpi = v_\mu$. We summarize this discussion in the following proposition. 

\begin{prop}
\label{prop:torus_computation}
Let $\X$ be a homogeneous space with projective embedding $\X\hookrightarrow\bP(V_\varpi)$ and a choice of reduced expression $\bs$ for $\wP$. The restriction of $p\in \C[\X]$ to the torus $\Xo\subset\X$ is a polynomial in the coordinates $(a_\pb)_{\pb\in\sH_{\bs}}$ of $\Xo$. In particular, the restriction to $\Xo$ induces an injection $\C[\X]\hookrightarrow \C[a_\pb ~|~ \pb\in\sH_{\bs}]$.
\end{prop}

Using this, we can now define the valuation $\nu_{\X,\varpi,\bs}$ as well as the Chevalley polytope $\cP_{\X,\varpi,\bs}$. 

\begin{df}
\label{def:main_objects}
    Let $(\X,\varpi,\bs)$ be a triple consisting of $\X$ a homogeneous space, $\varpi$ a dominant weight yielding a projective embedding $\X\hookrightarrow\bP(V_\varpi)$ and $\bs$ a choice of reduced expression for $\wP$. We define the valuation $\nu_{\X,\varpi,\bs}:\C[\X]\rightarrow\Z^{\ellwP}$ by sending $p\in\C[\X]$ to the exponent vector of the degree reverse lexicographic minimal term of the restriction $p|_{\Xo}$ with respect to the ordering $a_{\pb_{\ellwP}}>\dots>a_{\pb_1}$ of torus coordinates corresponding to the reduced expression $\bs$. Furthermore, we define the \emph{Chevalley polytope} $\cP_{\X,\varpi,\bs}$ to be the convex hull of the valuation of the homogeneous degree-$1$ part of $\C[\X]$, i.e.
    \[
    \cP_{\X,\varpi,\bs} = \conv\Bigl(\bigl\{\nu_{\X,\varpi,\bs}(p)~\big|~ p\in\C[\X],~\deg(p)=1\bigr\}\Bigr)~~\subset~~ \R^{\ellwP}.
    \]
\end{df}

\begin{rem}
\label{rem:ordering}
    While the definition above uses the degrevlex order on $\C[a_\pb~|~\pb\in\sH_{\bs}]$ to define the valuation $\nu_{\X,\varpi,\bs}$ and hence $\cP_{\X,\varpi,\bs}$, it is possible to make the definition above with respect to any term order on $\C[a_\pb~|~\pb\in\sH_{\bs}]$. Our reason for using the degrevlex order is that we will provide a combinatorial interpretation of $\nu_{\X,\varpi,\bs}$ in the following sections, and in that context there will be a natural (partial) term order on $\C[a_\pb~|~\pb\in\sH_{\bs}]$ which can be extended to the degrevlex order. Thus, the particular choice of degrevlex is not important, and the only constraint in selecting a term order is that it must extend the partial order which we will describe below.
\end{rem}

\section{Minuscule Chevalley polytopes}
\label{sec:minuscule}
In this section we study Chevalley polytopes in the case when $\X=\G/\P_k$ is a minuscule homogeneous space in its minimal (Pl\"ucker) embedding $\X\hookrightarrow \bP(V_{\fwt[k]})$ with $\fwt[k]$ a minuscule fundamental weight, and $\bs$ is any choice of reduced expression for $\wP$, where we write $\P=\P_k$. We will exploit the additional combinatorial tools arising from the minusculity of $\X$ and $V_{\fwt[k]}$ to prove that the Chevalley polytope $\cP_{\X,\fwt[k],\bs}$ is a Newton-Okounkov body for $\X$ and that the Pl\"ucker coordinates form a Khovanskii basis for $\X$, both with respect to the valuation $\nu_{\X,\fwt[k],\bs}$.

When $\X$ is minuscule, as noted in Section \ref{sec:minuscule_background}, the heap $\sH_{\bs}$ is independent of the chosen reduced expression $\bs$ for $\wP$ and hence the torus $\opendunim$ is also independent of $\bs$ (up to permutations of its coordinates). Furthermore, the heap $\sH_{\bs}$ coincides with the minuscule poset $\minposet$ corresponding to $V_\varpi$, so we refer to the torus coordinates by $a_{\pb}$ for $\pb\in\minposet$. Thus, the valuation $\nu_{\X,\fwt[k],\bs}$ and Chevalley polytope $\cP_{\X,\varpi,\bs}$ also do not depend on the choice of reduced expressions (up to permutations of the coordinates), and in the remainder of this section we will suppress the dependence of the valuation and Chevalley polytope on $\bs$ and simply write $\nu_{\X,\fwt[k]}$ and $\cP_{\X,\fwt[k]}$. 

We observed in \cite[equation (2.13)]{Spacek_Wang_cominuscule} (and it follows from the expansion in equation \eqref{eq:torus_expansion} above when restricted to minuscule spaces $\X$ and minuscule representations $V_{\fwt[k]}$) that the restrictions to $\Xo$ of Pl\"ucker coordinates on $\X$ can be written in terms of labeled poset embeddings of filters $\sF\subset\minposet$ into $\minposet$:
\[
p_{\sF}|_{\Xo} = \tsum[_{\emb:\,\sF\hookrightarrow\minposet}] a_{\emb},
\]
where $a_{\emb}=\prod_{\pb\in\sF} a_{\emb(\pb)}$ is the (square-free) monomial in the torus coordinates of $\Xo$ whose support is $\emb(\sF)$. Note that we used filters of $\minposet$ here instead of order ideals as in the reference \cite{Spacek_Wang_cominuscule}. This change is due to conventional differences resulting from the use of left quotients $\P\backslash\G$ in our previous work along with embedding $\P\backslash \G$ into a projective space defined over a dual representation. Since the Pl\"ucker coordinates form a set of algebra generators for $\C[\X]$, the restriction map thus provides an injection $\C[\X]\hookrightarrow\C[a_\pb~|~ \pb\in\minposet]$. 

This combinatorial description of the restrictions of Pl\"ucker coordinates induces a natural partial term order on $\C[a_\pb~|~\pb\in\minposet]$ as follows. For a filter $\sF\subset\minposet$ and two labeled embeddings $\emb_1,\emb_2:\sF\hookrightarrow\minposet$, we write $\emb_1\le\emb_2$ (and equivalently that $a_{\emb_1}\le a_{\emb_2}$) if for each $\pb\in\sF$, we have that $\emb_1(\pb) \ge \emb_2(\pb)$ (note the reversal!). Under this partial order, the identity embedding $\id:\sF\hookrightarrow\minposet$ is the unique minimal embedding of $\sF$ into $\minposet$. It is straightforward to check that the degrevlex order on $\C[a_\pb~|~\pb\in\minposet]$ with ordering of coordinates  $a_{\pb_{\ellwP}} > \ldots > a_{\pb_1}$ corresponding to the choice of $\bs$ extends this partial ordering. With this interpretation, we see that the valuation $\nu_{\X,\fwt[k]}$ sends the Pl\"ucker coordinate $p_{\sF}\in\C[\X]$ to the indicator vector of the (identity embedding of the) filter $\sF\subset\minposet$, so we obtain the following description for minuscule Chevalley polytopes:
\[
\cP_{\X,\fwt[k]} = \conv\bigl(\{\mathbf{1}_\sF\in\R^{\ellwP} ~|~ \text{$\sF\subset\minposet$ is a filter} \}\bigr),
\]
where $\mathbf{1}_\sF\in\R^{\ellwP}$ denotes the indicator vector. Thus, Proposition \ref{prop:order-vertices} immediately yields an equality between $\cP_{\X,\fwt[k]}$ and the order polytope $\cO_{\minposet}$ since they have the same set of vertices by Proposition \ref{prop:order-vertices}. We summarize this discussion in the following proposition:

\begin{prop}
    \label{prop:order_polytope}
    The Chevalley polytope $\cO_{\X,\fwt[k]}$ is equal to the order polytope $\cO_{\minposet}$ of $\minposet$.
\end{prop}

Interpreting the minuscule poset $\minposet$ as a skew Young diagram, linear extensions of $\minposet$ can then be interpreted as standard Young tableaux on $\minposet$. The equivalence is straightforward: given a linear extension of $\minposet$, we obtain a standard Young tableau on $\minposet$ by labeling $\pb\in\minposet$ with its position in the linear extension and conversely given a standard Young tableau on $\minposet$, we obtain a linear extension by ordering the elements in increasing order corresponding to their standard Young tableau labels. While it seems to be known generally that the number of such standard Young tableaux on $\minposet$ is equal to the degree of $\X$, we were not able to locate a satisfactory reference outside of type $\LGA$ so we include a proof here.

\begin{lem}
\label{lem:linear_extensions_volume}
The number of linear extensions of $\minposet$ is equal to the degree of $\X\hookrightarrow\bP(V_{\fwt[k]})$.
\end{lem}


\begin{proof}
The degree of $\X$ can be computed as the number of points in the intersection of $\X$ with a generic codimension-$\ellwP$ hyperplane, where $\ellwP=\dim(X)=\ell(\wP)$. Equivalently, in cohomological terms, the degree of $\X$ can be computed as the the coefficient of the class of a point $\sigma_{\minposet}$ in the $\ellwP$th power of the hyperplane class $\sigma_{\Yboxdim{3pt}\yng(1)}$. As remarked in \cite[Section 2.2]{CMP_Quantum_cohomology_of_minuscule_homogeneous_spaces}, this quantity can be computed for minuscule $\X$ as the number of maximal chains (totally ordered subsets) in the poset of Schubert varieties of $\X$ between the hyperplane class $\sigma_{\Yboxdim{3pt}\yng(1)}$ and the class of a point $\sigma_{\minposet}$, using the \emph{Chevalley formula} for multiplication by the hyperplane class.

As shown e.g.~in \cite[V,\S3]{hiller}, the poset of Schubert varieties of $\X$ is isomorphic to the set of minimal coset representatives $\cosets = \weyl/\weylp$ with the Bruhat order. Since $V_{\fwt[k]}$ is a minuscule representation, the set of weights of $V_{\fwt[k]}$ forms a unique orbit under the action of the Weyl group $\weyl$ on the highest weight $\fwt[k]$. Furthermore, the highest weight $\fwt[k]$ is stabilized by the parabolic Weyl subgroup $\weylp$, so that we may identify the weight poset of $V_{\fwt[k]}$ with the set $\cosets$ of minimal coset representatives with the Bruhat order. Thus, we may compute the degree of $\X$ as the number of maximal chains in poset of weights of $V_{\fwt[k]}$. It is a standard fact in lattice theory (see e.g.~\cite[p. 198]{stanley_supersolvable}) that the number of maximal chains of a distributive lattice $\sL$ is exactly equal to the number of linear extensions of the associated poset $\sP$ of join-irreducible elements of $\sL$. In our case, this means that the number of maximal chains of the poset of weights of the minuscule representation $V_{\fwt[k]}$ is equal to the number of linear extensions of $\minposet$, from which we conclude immediately that $\deg(\X)$ is the number of linear extensions of $\minposet$. 
\end{proof}

This lemma provides a combinatorial description of the degree of $\X$ in its projective embedding $\X\hookrightarrow\bP(V_{\fwt[k]})$, and with this we can now relate the degree of $\X$ directly to the volume of the order polytope $\cO_{\minposet}$ and hence the volume of the Chevalley polytope $\cP_{\X,\fwt[k]}$. 

\begin{cor}
\label{cor:volumes}
The degree of $\X\hookrightarrow\bP(V_{\fwt[k]})$ is equal to the (normalized) volume of the Chevalley polytope $\cP_{\X,\fwt[k]}$.
\end{cor}

\begin{proof}
By Proposition \ref{prop:order_polytope}, $\cP_{\X,\fwt[k]}$ is the order polytope $\cO_{\minposet}$ of $\minposet$, so we have by Proposition \ref{prop:stanley_volume} that $\vol(\cP_{\X,\fwt[k]})$ is equal to the number of linear extensions of $\minposet$. By Lemma \ref{lem:linear_extensions_volume}, the number of linear extensions of $\minposet$ is equal to $\deg(\X)$, so the desired equalities follow immediately.
\end{proof}

It is clear from the definitions of $\cP_{\X,\fwt[k]}$ and $\Delta(\C[\X],\nu_{\X,\fwt[k]})$ that $\cP_{\X,\fwt[k]}\subset \Delta(\C[\X],\nu_{\X,\fwt[k]})$, since the vertices of $\cP_{\X,\fwt[k]}$ are the valuations of Pl\"ucker coordinates (which are homogeneous elements of $\C[\X]$ of degree $1$) and are thus contained in $\Delta(\C[\X],\nu_{\X,\fwt[k]})$. We now establish the reverse inclusion.

\begin{thm}
  \label{thm:nobody}
  Let $\X$ be a minuscule space in its minimal embedding $\X\hookrightarrow\bP(V_{\fwt[k]})$. The Newton-Okounkov body $\Delta(\C[\X],\nu_{\X,\fwt[k]})$ is equal to the Chevalley polytope $\cP_{\X,\fwt[k]}$.
\end{thm}

\begin{proof}
    Restriction to $\Xo\subset\X$ induces an injection $\C[\X]\hookrightarrow\C[a_\pb~|~\pb\in\minposet]$ which identifies $\C[\X]$ with a finitely generated subalgebra of a polynomial ring and $\nu_{\X,\fwt[k]}$ with the valuation sending an element to its degrevlex minimal term. Thus, by Lemma \ref{lem:one-dim-leaves}, we obtain that $\nu_{\X,\fwt[k]}$ has one-dimensional leaves, so Proposition \ref{prop:no-volume} implies that the normalized volume of the associated Newton-Okounkov body $\Delta(\C[\X],\nu_{\X,\fwt[k]})$ is equal to the degree of $\X$.

    By definition of $\Delta(\C[\X],\nu_{\X,\fwt[k]})$, the rescaled valuation of any homogeneous element of $\C[\X]$ is contained in the Newton-Okounkov body $\Delta(\C[\X],\nu_{\X,\fwt[k]})$, and therefore the same holds for the convex hull of the rescaled valuations of any finite subset of homogeneous elements of $\C[\X]$. Thus, $\cP_{\X,\fwt[k]}\subset\Delta(\C[\X],\nu_{\X,\fwt[k]})$. Finally, Corollary \ref{cor:volumes} implies that $\vol(\cP_{\X,\fwt[k]})=\deg(\X)$ and since both $\cP_{\X,\fwt[k]}$ and $\Delta(\C[\X],\nu_{\X,\fwt[k]})$ are closed convex sets, we obtain $\cP_{\X,\fwt[k]} = \Delta(\C[\X],\nu_{\X,\fwt[k]})$.
\end{proof}

In particular, $\Delta(\C[\X],\nu_{\X,\fwt[k]})$ is a lattice polytope (in fact a $0/1$-polytope) with vertices given the valuations of Pl\"ucker coordinates on $\X$. Moreover, it now follows immediately from Corollary \ref{cor:no_lattice_pts} that $\Delta(\C[\X],\nu_{\X,\fwt[k]})$ also has no lattice points other than vertices.

\begin{cor}
    \label{cor:NObody_no_lattice_pts}
    $\Delta(\C[\X],\nu_{\X,\fwt[k]})$ has no lattice points other than its vertices.
\end{cor}

\begin{proof}
    By Theorem \ref{thm:nobody}, we have $\Delta(\C[\X],\nu_{\X,\fwt[k]})=\cP_{\X,\fwt[k]}$, but the latter has no lattice points other than its vertices by Corollary \ref{cor:no_lattice_pts}, so the desired statement follows immediately.
\end{proof}

Finally, we conclude this section by showing that the Pl\"ucker coordinates on $\X$ form a Khovanskii basis for $\C[\X]$ with respect to $\nu_{\X,\fwt[k]}$.

\begin{thm}
    \label{thm:plucker-khovanskii}
    Let $\X$ be a minuscule space in its minimal embedding $\X\hookrightarrow\bP(V_{\fwt[k]})$. The Pl\"ucker coordinates on $\X$ form a Khovanskii basis for $\C[\X]$ with respect to the valuation $\nu_{\X,\fwt[k]}$, and $\X$ admits a degeneration to the projectively normal toric variety associated to $\cP_{\X,\fwt[k]}$.
\end{thm}

\begin{proof}
By \cite[Proposition 2.4]{Kaveh_Manon_Khovanskii_bases}, it suffices to prove that the valuations of Pl\"ucker coordinates $\{\nu_{\X,\fwt[k]}(p_{\sF})\mid\sF\subset\minposet\}$ generate the \emph{value semigroup} $S(\C[\X],\nu_{\X,\fwt[k]})=\{\nu_{\X,\fwt[k]}(f)\mid f\in \C[\X]\}$. Since $\C[\X]$ is $\Z_{\ge0}$-graded, it further follows from \cite[Definitions 2.20, 2.21]{Kaveh_Manon_Khovanskii_bases} that $S(\C[\X],\nu_{\X,\fwt[k]}) = \bigcup_{i\ge 0} \nu_{\X,\fwt[k]}(\C[\X]_i\setminus\{0\})$, so it suffices to show that the valuation of any homogeneous element of $\C[\X]$ can be obtained from the valuations of Pl\"ucker coordinates. 

Thus, we consider $f\in \C[\X]$ homogeneous of degree $m$, then by definition of $\Delta(\C[\X],\nu_{\X,\fwt[k]})$ we have that $\frac{1}{m}\nu_{\X,\fwt[k]}(f)\in \Delta(\C[\X],\nu_{\X,\fwt[k]})$ so that $\nu_{\X,\fwt[k]}(f)\in m\Delta(\C[\X],\nu_{\X,\fwt[k]})\cap\Z^{\ellwP}$. Since Proposition \ref{prop:order_polytope} says that $\Delta(\C[\X],\nu_{\X,\fwt[k]})=\cP_{\X,\fwt[k]}$ is the order polytope of $\minposet$, it follows from the integer decomposition property established in Corollary \ref{cor:idp} that $\nu_{\X,\fwt[k]}(f)$ can be written as a sum of $m$ vertices of $\cP_{\X,\fwt[k]}$, which are exactly the valuations of Pl\"ucker coordinates. Thus, the valuations of Pl\"ucker coordinates generate the semigroup $S(\C[\X],\nu_{\X,\fwt[k]})$ and hence the Pl\"ucker coordinates form a Khovanskii basis for $\X$ with respect to $\nu_{\X,\fwt[k]}$. 

Since we work with projective embeddings of $\X$ of the form $\X\hookrightarrow \bP(V_{\fwt[k]})$ arising as the closed orbit of the highest weight line in a finite-dimensional, irreducible $\G$-representations, our $\X$ are always projectively normal. Thus, the conditions of \cite[Proposition 17.4]{Rietsch_Williams_NO_bodies_cluster_duality_and_mirror_symmetry_for_Grassmannians} are satisfied in our case, and we obtain the desired toric degeneration.
\end{proof}

Newton-Okounkov bodies are generally far from polyhedral and in many cases Newton-Okounkov bodies and Khovanskii bases are difficult to construct, so the ease of our construction above for minuscule spaces provides a large family of polytopal Newton-Okounkov bodies with good combinatorial properties arising from the minuscule poset $\minposet$. 

\section{General Chevalley polytopes}
\label{sec:general}
While our proofs in the previous section using the minuscule assumption on $\X$ relied on the combinatorics of minuscule posets, Definition \ref{def:main_objects} applies equally well to any $\X$, any dominant weight $\varpi$ inducing a projective embedding $\X\hookrightarrow\bP(V_\varpi)$, and any reduced expression $\bs$ for $\wP$. The first difference in the general case is that the heap $\sH_{\bs}$ now depends on the choice of $\bs$. As a result, the valuation and torus also depend on the choice of $\bs$, and we lose access to the combinatorics of order ideals and order polytopes. As a first step towards generalizing the tools in the minuscule case, we give a combinatorial interpretation of the valuations of Pl\"ucker coordinates which generalizes the minuscule combinatorics. 

We define a \emph{weighted heap} $(\sH,\weight)$ to be a heap $\sH$ together with a weight function $\weight:\sH\rightarrow\N$. Now, fix a weight basis $\cB$ for $V_\varpi$, a dual basis $\cB^*$ for $V_\varpi^*$, and let $v_\mu\in\cB$ be a vector of weight $\mu$. Then for $\boldf=\Chf_{i_\ell}^{d_\ell}\cdots\Chf_{i_1}^{d_1}$, where $d_i\ge 0$, such that $v_\mu^*(\boldf\cdot v_\varpi) \neq 0$ and such that $i_{j+1}\neq i_j$ for all $j$, we define the weighted heap $(\sH_{\boldf},\weight_{\boldf})$ consisting of the heap $\sH_{\boldf}$ on the elements $\pb_i$ with $d_i>0$ with partial order $\pb_j < \pb_k$ if $j>k$ and $\Chf_{i_j}$ and $\Chf_{i_k}$ do not commute and weight function $\weight_{\boldf}(\pb_i) = d_i$. Then we define the set $\sH_{v_\mu}$ to be the set of all weighted heaps arising from such $\boldf$: 
\[\sH_{v_\mu} = \{(\sH_{\boldf},\weight_\boldf)\mid v_\mu^*(\boldf\cdot v_\varpi)\neq 0\}.\]

\begin{rem}
    In the above definition of $\sH_{v_\mu}$, we may replace $v_\mu^*(\boldf\cdot v_\varpi)=0$ with the simpler condition $\boldf\cdot v_\varpi = v_\mu$ (up to scaling) when there is a unique choice of Chevalley basis (also up to scaling) for $V_\varpi$. This is true when $V_\varpi$ is minuscule, which greatly simplifies the definitions used in the previous section.
\end{rem}

Given a weighted heap $(\sH,\weight)$, we let $\weight\cdot\sH$ denote the poset consisting of $\weight(\pb)$ incomparable copies $\pb_1,\dots, \pb_{\weight(\pb)}$ of each $\pb\in\sH$ with the partial order $\pb_i < \pb_j'$ if $\pb < \pb'$. As with $\minposet$, each element of a heap can be labeled by a simple reflection. Then a \emph{labeled weighted embedding} $\emb:(\sH,\weight)\hookrightarrow\sH_{\bs}$ is a labeled poset morphism $\emb:\weight\cdot\sH\hookrightarrow\sH_{\bs}$. The image is a weighted subset $(\emb(\weight\cdot\sH), \weight_{\emb})$ of $\sH_{\bs}$, where $\weight_{\emb}(\pb) = |\emb^{-1}(\pb)|$ for every $\pb\in\emb(\weight\cdot(\sH))$. (This generalizes the definitions of weighted embeddings of weighted order ideals in \cite[Definitions 2.20, 2.21]{Spacek_Wang_cominuscule}.)

Using these definitions, we can now rewrite the restrictions of the coordinate functions $p_{v_\mu}\in\C[\X]$ corresponding to weight vectors in our fixed basis $v_\mu\in\cB$ as
\[
p_{v_\mu}|_{\Xo} = \sum_{(\sH,\weight)\in\sH_{v_\mu}} \sum_{\emb:(\sH,\weight)\hookrightarrow \sH_{\bs}} a_\emb,
\qwhere
a_\emb = \prod_{\pb\in\emb(\weight\cdot\sH)} \frac{a_{\pb}^{\weight_\emb(\pb)}}{\weight_\emb(\pb)!}.
\]
Note that it is possible there are no embeddings of a given weighted heap $(\sH,\weight)$ into $\sH_{\bs}$, in which case the second sum above is empty. Since these coordinates form an algebra generating set for $\C[\X]$, the valuations of arbitrary polynomials in the $p_{b}$ for $b\in\cB$ can be computed using the expressions above.

The main obstacle in generalizing our tools from the minuscule case lies in the fact that the weighted heaps arising above generally are not filters of $\sH_{\bs}$ and furthermore need not even contain the same sets of labels due to braid moves. However, there is still a natural partial order on the set of weighted heaps $\sH_{v_\mu}$ and their embeddings, and we believe that there should correspondingly be a notion of \emph{heap polytope} which generalizes the order polytopes of minuscule posets. We now conjecture the generalizations of the statements observed in the minuscule case, and we plan to investigate them in upcoming work.

\begin{conj}
\label{conj:nobody}
    Let $\X=\G/\P$ be a homogeneous space together with a projective embedding $\X\hookrightarrow \bP(V_\varpi)$. There exists at least one reduced expression $\bs$ for $\wP$ such that $\cP_{\X,\varpi,\bs}$ is a Newton-Okounkov body for $\X$ with respect to the valuation $\nu_{\X,\varpi,\bs}$.
\end{conj}


Note that unlike the minuscule case, in general certain Chevalley polytopes $\cP_{\X,\varpi,\bs}$ may be strictly contained in the Newton-Okounkov body for $\X$ with respect to $\nu_{\X,\varpi,\bs}$. It would be interesting to understand more concretely which reduced expressions are distinguished by having the corresponding Chevalley polytope equal to a Newton-Okounkov body for $\X$.

\begin{conj}
\label{conj:khovanskii}
Let $\X=\G/\P$ be a homogeneous space together with a dominant weight $\varpi$ yielding a projective embedding $\X\hookrightarrow\bP(V_\varpi)$. There exists a reduced expression $\bs$ for $\wP$ such that the set of coordinates corresponding to any Chevalley basis for $V_\varpi$ forms a Khovanskii basis for $\C[\X]$, and $\X$ admits a toric degeneration whose normalization is the toric variety associated to $\cP_{\X,\varpi,\bs}$.
\end{conj}

In some cases, there will be a unique Chevalley basis for $V_\varpi$ up to scaling. However, in general, there will be many choices of a Chevalley basis, and we expect that any such choice will satisfy the previous conjecture. The dependence of Conjecture \ref{conj:khovanskii} on the particular choice of weight basis is due to the fact that the valuations of linearly independent elements may fail to be distinct. Thus, finding a Khovanskii basis for $\C[\X]$ depends not only on defining a suitable valuation but also finding a suitable set of functions whose corresponding valuations also generate the valuation semigroup. Finally, we have also observed the following remarkable property of Chevalley polytopes:

\begin{conj}
\label{conj:decomposition}
    Let $\X=\G/\P\hookrightarrow\bP(V_\varpi)$ be a homogeneous space with a projective embedding, and let $\varpi=\sum_i c_i\fwt[i]$ be a dominant integral weight for $\G$ written as a linear combination of fundamental weights. Then, for any reduced decomposition $\bs$ of $\wP$, there exist reduced decompositions $\bs_j$ of the minimal coset representatives $w^{P_j}$ such that 
    \[
    \cP_{\X,\varpi,\bs} = \sum\nolimits_j c_j\cdot \cP_{\G/\P_j, \fwt[j], \bs_j},
    \]
    where the sum denotes the Minkowski sum in $\R^{\ellwP}$ with each $\cP_{\G/\P_j,\fwt[j],\bs_j}$ re-embedded 
    into $\R^{\ellwP}$ by identifying of $w^{P_i}$ as a subword of $\wP$. 
\end{conj}

We view this conjecture as providing a polyhedral shadow of the Littlewood-Richardson rule: the lattice points of Chevalley polytopes are parametrized by a weight basis for $V_{\fwt}$ and thus the Minkowski sum of these polytopes corresponds to a (distinguished subrepresentation of) the tensor product 
\[
V_\varpi\subset \bigotimes\nolimits_i V_{\fwt}^{\otimes c_i}.
\]
Other subrepresentations in the decomposition of the tensor product can be seen as polytopes which are strictly contained in $\cP_{\X,\varpi,\bs}$.

\section{Comparison with string polytopes}
\label{sec:string}
Our polytopes share many similarities with \emph{string polytopes}, originally studied by Littelman and Bernstein-Zelevinsky \cite{littelman_cones_crystals_patterns, BZ-tensor-product-mult}. We give a brief introduction of string polytopes following \cite[Section 3]{kaveh_strings}, and we refer to the references therein for a more detailed exposition. First, the definition of string polytope requires the notion of a \emph{crystal basis} for $V_\varpi$, which is a combinatorially defined weight basis chosen as follows. For any simple root $\sr_i$, we define functions $V_\varpi\rightarrow\Z$: 
\begin{align*}
    \epsilon_i(v) &= \max\{a~|~\Che_i^a \cdot v\neq 0\}~\textrm{and}\\
    \varphi_i(v) &= \max\{a~|~\Chf_i^a \cdot v\neq 0\},
\end{align*}
where we set $\epsilon_i(v)=0$ if $\Che_i\cdot v = 0$ and analogously for $\varphi_i(v)$. 

\begin{df}[{\cite[Definition 3.1]{kaveh_strings}}]
\label{df:crystal_basis}
For a weight basis $\cB$ of $V_\varpi$ and $b\in\cB$, let the action of $\Che_i$ on $b$ be given by $\Che_i\cdot b = \sum_j c_j b_j$ and analogously $\Chf_i\cdot b = \sum_j d_j b_j$. Then $\cB$ is a \emph{crystal basis} for $V_\varpi$ if:
\begin{enumerate}
    \item for every $j$ such that $c_j\neq 0$, we have $\epsilon_i(b_j) < \epsilon_i(b)$;
    \item if $\epsilon_i(b)\neq 0$, then there is a unique $j$ such that $\epsilon_i(b_j) = \epsilon_i(b)-1$, and for this $j$ we further have $\varphi_i(b) = \varphi_i(b_j)-1$;
    \item for every $j$ such that $d_j\neq 0$, we have $\varphi_i(b_j) < \varphi_i(b)$; and
    \item if $\varphi_i(b)\neq 0$, then there is a unique $j$ such that $\varphi_i(b_j) = \varphi_i(b)-1$, and for this $j$ we further have $\epsilon_i(b) = \epsilon_i(b_j)-1$.
\end{enumerate}
\end{df}

The existence of a weight basis satisfying such properties was proven in \cite[Lemma 5.50]{BK-crystals}, and such a basis is sometimes also called a \emph{perfect basis} or a \emph{canonical basis}. The \emph{crystal graph} associated to a crystal basis $\cB$ is the graph with vertex set $\cB$ and edges $(b,b')$ if $\epsilon_i(b')=\epsilon_i(b)-1$ and $\Che_i\cdot b$ has $b'$ with nonzero coefficient in its $\cB$ expansion (or equivalently, if $\varphi_i(b)=\varphi_i(b')-1$ and $\Chf_i\cdot b'$ has $b$ with nonzero coefficient in its $\cB$ expansion). Remarkably, the crystal graph does not depend on the choice of $\cB$. 

Next we define the \emph{string parametrization} $\iota_0(\cB)$ of a crystal basis $\cB$ associated to a reduced expression ${\bf s}_0=s_{r_1}\cdots s_{r_{\ell(\wo)}}$ for the longest word $\wo\in\weyl$ by associating to each basis vector $v\in \cB$ the integer vector $\iota_0(v)=(a_1,\dots, a_{\ell(\wo)})$ computed iteratively via
\begin{align*}
    a_1 &= \max \{a~|~\tilde{\Che}_{r_1}^a(v)\neq 0\}\\
    a_2 &= \max \{a~|~\tilde{\Che}_{r_2}^a\tilde{\Che}_{r_1}^{a_1}(v)\neq 0\}\\
    &\vdots\\
    a_{\ell(\wo)} &= \max \{a~|~\tilde{\Che}_{r_{\ell(\wo)}}^a\cdots \tilde{\Che}_{r_1}^{a_1}(v)\neq 0\},
\end{align*}
where $\tilde{\Che}_{r_i}(b)$ is the unique element $b'\in\cB$
such that $\epsilon_{r_i}(b)=\epsilon_{r_i}(b')-1$. (Note that we used $\tilde{\Che}_{r_i}$ here as opposed to $\tilde{\Chf}_{r_i}$ as was presented in \cite[Section 3.1]{kaveh_strings}. This is because we prefer to work in $V_\varpi$ rather than its dual.) 

For each dominant integral weight $\varpi$, we fix a crystal basis $\cB_\varpi$ for the corresponding finite-dimensional irreducible representation $V_\varpi$, though the image of the string parametrization only depends on the crystal graph and not the particular choice of crystal basis. Then collecting the string parametrizations of crystal bases $\cB_\varpi$ for all finite-dimensional irreducible $\G$-representations $V_\varpi$, we define the \emph{string cone} 
\[
C_{{\bf s}_0}=\overline{\mathrm{cone}\left(\bigcup_{\varpi\in\Lambda^+} \{\varpi\}\times \iota_0(\cB_\varpi)\right)}
\]
where $\Lambda^+$ denotes the set of dominant integral weights. For a particular $\varpi\in\Lambda^+$, we further define the \emph{string polytope}
\[
\cS_{\varpi,\bs_0} = C_{{\bf s}_0}\cap\bigl(\{\varpi\}\times\R^{\ell(\wo)}\bigr).
\]

The main difference between the string parametrization and our valuations arises from our choice of action. In the context of the representation $V_\varpi$, our construction ``starts from'' the highest weight vector $v_\varpi$ and acts by a monomial in the $\Chf_i$ distinguished by the choice of $\bs$ to obtain a given weight vector $v_\mu$ for the construction of our Chevalley polytopes, while the string parametrization instead ``starts from'' from a crystal basis vector $v_\mu$ and identifies a path in the crystal graph on $\cB$ distinguished by the choice of $\bs_0$. Although the difference is somewhat subtle, we observed that our approach often yields polytopes possessing much better combinatorial properties, as we now illustrate.

\begin{ex}
We explicitly compute the Chevalley and string polytopes for the maximal orthogonal Grassmannian $\X=\OG(2,5)=\G/\P$ for $\G=\Spin_5$ of type $\LGB_2$ with $\P=\P_2=\P_{\fwt[2]}$, with the minimal embedding $\OG(2,5)\hookrightarrow\bP(V_{\fwt[2]})$ obtained from the fundamental weight $\fwt[2]$, and with the reduced expression $\bs_0=s_2s_1s_2s_1$ for $\wo$ and $\bs=s_2s_1s_2$ for $\wP$. Since this representation is minuscule, all the weight spaces are one-dimensional, leading to a canonical choice of crystal basis up to (global) scaling. The fundamental representation $V_{\fwt[2]}$ is $4$-dimensional, and we compute the string parametrizations:
\ali{
\iota_0(v_{\fwt[2]}) &= (1,1,1,0), \qquad \iota_0(v_{\fwt[1]-\fwt[2]}) = (0,1,1,0), \\
\iota_0(v_{-\fwt[1]+\fwt[2]}) &= (1,0,0,0),\qquad \iota_0(v_{-\fwt[2]}) = (0,0,0,0).}
These four points are lattice points of the string polytope $\cS_{\fwt[2],\bs_0}$. However, these four points line on a plane since $(1,0,0,0) + (0,1,1,0) = (1,1,1,0)$ and are the vertices of a $2$-dimensional square, which cannot be a Newton-Okounkov body for $\X=\LGB_2/\P_2$ since it does not have dimension $3=\ellwP=\dim(\X)$ and volume $1=\deg(\X)$. In fact, this square is not even equal to the string polytope $\cS_{\fwt[2],\bs_0}$. In \cite[Example 6.2.9]{steinert_thesis}, Steinert shows that $\cS_{\fwt[2],\bs_0}$ has one further vertex: 
\[
(0,0,0,0),\qquad (0,\tfrac{1}{2},1,0),\qquad (0,1,1,0),\qquad (1,0,0,0),\qquad(1,1,1,0),
\]
and in particular is not integral or full-dimensional. Moreover, $\cS_{\fwt[2],\bs_0}$ has five vertices, whereas the representation $V_{\fwt[2]}$ is $4$-dimensional, so it is unclear what meaning to assign the final vertex. Furthermore, $\cS_{\fwt[2],\bs_0}$ does not satisfy the integer decomposition property because the lattice point $(0,1,2,0) = 2\cdot(0,\frac{1}{2},1,0)$ cannot be written as a sum of lattice points of $\cS_{\fwt[2],\bs_0}$ (note that $(0,\frac{1}{2},1,0)$ is not a lattice point, even though it is a vertex). Despite these drawbacks, the convex hull of these vertices is $3$-dimensional and has (normalized) volume $1$, so the string polytope $\cS_{\fwt[2],\bs_0}$ is in fact a Newton-Okounkov body for $\X$. 

On the other hand, since $\fwt[2]$ is a minuscule fundamental weight in type $\LGB_2$, our Chevalley polytope $\cP_{\X,\fwt[2]}$ does not depend on the reduced expression for $\wP$ and can be constructed directly by considering order filters of the minuscule poset $\minposet=\raisebox{-2pt}{\Yboxdim{6pt}\tiny\young(2,12)}$. This poset has the filters $\varnothing$, $\yng(1)$, $\yng(1,1)$, $\yng(1,2)$, corresponding to the indicator vectors
\[
(0,0,0),\qquad(1,0,0),\qquad(1,1,0),\qquad(1,1,1).
\]
The Chevalley polytope $\cP_{\X,\fwt[2]}$ defined by these vertices is a full-dimensional $3$-simplex with volume $1$ and is a Newton-Okounkov body for $\X$. Furthermore, since it is the order polytope of $\minposet$, it has the integer decomposition property by Corollary \ref{cor:idp}, is an integral 0/1-polytope, and parametrizes a weight basis for the $4$-dimensional representation $V_{\fwt[2]}$, in contrast with the string polytope above.
\end{ex}

This behavior of string polytopes is not an isolated occurrence. Steinert notes in \cite[Section 6.2.1]{steinert_thesis} the existence of an infinite family of string polytopes associated to the minuscule type-$\LGA$ Grassmannians $\Gr(3,n)$ which do not have integral vertices or the integer decomposition property. Hence, even for minuscule $\X$ in their minimal embeddings $\X\hookrightarrow\bP(V_{\fwt[k]})$, the string polytopes often fail to be well-behaved if the reduced decomposition is not carefully chosen. On the other hand, our minuscule Chevalley polytopes are \emph{always} integral polytopes whose lattice points parametrize a weight basis for $V_{\fwt[k]}$ and which have good combinatorial properties arising from order polytopes, such as the integer decomposition property. (We note that Steinert suggests in \cite[Remark 6.2.7]{steinert_thesis} that the reason for the bad behavior of certain reduced decompositions is the failure of the integer decomposition property.) Thus, despite the similarities in the definitions of string polytopes and Chevalley polytopes, our Chevalley polytopes possess much better combinatorial properties in the minuscule case, and we believe this behavior extends to the non-minuscule case as well.
\bibliographystyle{amsalpha}
\bibliography{bib}
\addresseshere

\newpage
\appendix
\section{Non-minuscule examples}
\label{sec:examples}
Here, we present non-minuscule homogeneous spaces considered in various embeddings for which Conjectures \ref{conj:nobody}, \ref{conj:khovanskii}, and \ref{conj:decomposition} hold. Our proofs in the minuscule case relied on computing the volume of the Chevalley polytope $\cP_{\X,\varpi,\bs}$, which is by definition a subset of the Newton-Okounkov body for $\X$ with respect to $\nu_{\X,\varpi,\bs}$, and show that it is equal to the degree of $\X$. While we will develop tools for performing this computation in general in upcoming works, we now give several examples where the same result holds. We computed the degree of $\X$ using linear extensions of $\minposet$ in the minuscule case, and in general we use the following formula of Borel and Hirzebruch \cite[Theorem 24.10]{BH}, as cited in \cite{GW_hilbert}:
\[
\deg(\X,V_\varpi) = \dim(\X)!\prod_{\alpha}\frac{\llan\sdr,\varpi\rran}{\llan\sdr,\rho\rran},
\]
where $\rho$ denotes the half sum of all positive roots and the product is over positive roots such that $\llan\sdr,\varpi\rran\neq0$. Furthermore, we used Polymake \cite{polymake-2000, polymake-2017} to compute the volumes of the polytopes below to verify whether or not they are Newton-Okounkov bodies for the corresponding $\X$. Finally, we used Polymake to check in each case whether the lattice points of each polytope $\cP$ below form a Hilbert basis for the cone over $\{1\}\times\cP$, which indicates that the polytope has the integer decomposition property and hence that the corresponding coordinates form a Khovanskii basis for $\C[\X]$ as well as providing a projectively normal toric degeneration.

\addtocontents{toc}{\SkipTocEntry}
\subsection{The adjoint two-step flag variety}
The first example we present is the two-step partial flag variety $\X=\Fl(1,3;4)=\SL_4/\P_{1,3}$ where all three conjectures apply and hold. We note that $\varpi=\fwt[1]+\fwt[3]$ is an adjoint weight (and therefore a coadjoint weight as well, since $\LGA_3$ is simply-laced), and that the same holds for any $\fwt[1]+\fwt[n]$ in type $\LGA_n$. 

We will consider the minimal embedding of $\X=\Fl(1,3;4)$ into $\bP(V_\varpi)$ for $\varpi=\fwt[1]+\fwt[3]$. According to Conjecture \ref{conj:decomposition}, we should compare the polytope $\cP_{\X,\varpi,\bs}$ for a suitable choice of reduced expression $\bs$ for $\wP$ with the polytopes $\cP_1=\cP_{\Gr(1,4),\fwt[1],s_3s_2s_1}$ and $\cP_3=\cP_{\Gr(3,4),\fwt[3],s_1s_2s_3}$; note that these latter two polytopes are described by the minuscule construction in Section \ref{sec:minuscule} as the fundamental weights $\fwt[1]$ and $\fwt[3]$ are minuscule. Let $(1,2,3,4)$ and $(\bar4,\bar3,\bar2,\bar1)$ denote the weight bases of the minuscule fundamental representations $V_{\fwt[1]}$ and $V_{\fwt[3]}$ respectively. 

It is easy to see that the minuscule posets $\sw^{\P_1}=\mbox{\Yboxdim{6pt}\tiny\young(123)}$ and $\sw^{\P_3}=\mbox{\Yboxdim{6pt}\tiny\young(321)}$ are isomorphic, so that we obtain the identical (considered as subsets of $\R^3$) polytopes $\cP_1\cong\cP_3$ formed by the convex hull of $(0,0,0)$, $(1,0,0)$, $(1,1,0)$ and $(1,1,1)$. To obtain a Chevalley basis for $V_\varpi$, we use that $V(\fwt[1])\ot V(\fwt[3])=V(\varpi)\op\C$; writing $i\bar\jmath=i\ot\bar\jmath$, then the trivial subrepresentation is generated by the vector $1\bar1-2\bar2+3\bar3-4\bar4$. Since $\LGA_n$ is simply-laced and $\varpi$ is an adjoint weight, all non-zero weights are extremal (i.e.~in the Weyl orbit of $\varpi$) and the corresponding weight spaces are therefore one-dimensional; the remaining three-dimensional subspace is the zero weight space, naturally spanned by the three vectors
\[
a=\Chf_1(1\bar2)=1\bar1+2\bar2, \quad b=\Chf_2(2\bar3)=2\bar2+3\bar3 \qand c=\Chf_3(3\bar4)=3\bar3+4\bar4.
\]
Below we have drawn the actions by the Chevalley generators on the Chevalley bases of the representations $V_{\fwt[1]}$, $V_{\fwt[3]}$ and $V_{\varpi}$ respectively. Outside of minuscule representations, there may be non-trivial actions of Chevalley generators where the corresponding simple reflections act trivially. This occurs for the zero weight spaces in this example, which we have indicated by ``hollow'' circles, and we indicate the corresponding action using ``hollow'' lines. Note that the ``hollow'' lines are only ``one-way'' leaving the zero weight spaces, e.g.
\[\Chf_3(2\bar2 + 3\bar3) = 4\bar3\qquad\textrm{while}\qquad \Che_3(4\bar3) = 3\bar3 + 4\bar4 \neq 2\bar2+3\bar3,\]
so the hollow edge between $b$ and $4\bar3$ below should be thought of as oriented from $b$ to $4\bar3$. 
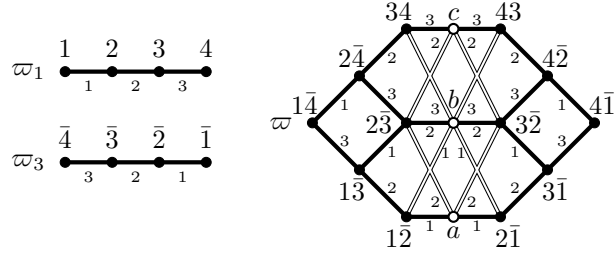
\[
\begin{array}{c}
    \begin{tikzpicture}[scale=0.625]
    \coordinate (1) at (-1.5,0);
    \coordinate (2) at (-0.5,0);
    \coordinate (3) at (0.5,0);
    \coordinate (4) at (1.5,0);
    \draw[ultra thick]
        (1)--(4)
    ;
    \draw[black,fill=black]
        (1) circle (3pt)
        (2) circle (3pt)
        (3) circle (3pt)
        (4) circle (3pt)
    ;
    \node at (1) [left=4pt] {$\fwt[1]$};
    \node at (1) [above=2pt] {$1$};
    \node at (2) [above=2pt] {$2$};
    \node at (3) [above=2pt] {$3$};
    \node at (4) [above=2pt] {$4$};
    \node at (-1,0) [below] {\tiny$1$};
    \node at (0,0) [below] {\tiny$2$};
    \node at (1,0) [below] {\tiny$3$};
\end{tikzpicture} \\[3pt]
    \begin{tikzpicture}[scale=0.625]
    \coordinate (1) at (-1.5,0);
    \coordinate (2) at (-0.5,0);
    \coordinate (3) at (0.5,0);
    \coordinate (4) at (1.5,0);
    \draw[ultra thick]
        (1)--(4)
    ;
    \draw[black,fill=black]
        (1) circle (3pt)
        (2) circle (3pt)
        (3) circle (3pt)
        (4) circle (3pt)
    ;
    \node at (1) [left=4pt] {$\fwt[3]$};
    \node at (1) [above=2pt] {$\bar4$};
    \node at (2) [above=2pt] {$\bar3$};
    \node at (3) [above=2pt] {$\bar2$};
    \node at (4) [above=2pt] {$\bar1$};
    \node at (-1,0) [below] {\tiny$3$};
    \node at (0,0) [below] {\tiny$2$};
    \node at (1,0) [below] {\tiny$1$};
\end{tikzpicture}      
\end{array}\quad
\begin{tikzpicture}[scale=0.625,baseline=0em]
    \coordinate (14) at (-3,0);
    \coordinate (24) at (-2,1);
    \coordinate (13) at (-2,-1);
    \coordinate (12) at (-1,-2);
    \coordinate (23) at (-1,0);
    \coordinate (34) at (-1,2);
    \coordinate (3344) at (0,2);
    \coordinate (2233) at (0,0);
    \coordinate (1122) at (0,-2);
    \coordinate (43) at (1,2);
    \coordinate (32) at (1,0);
    \coordinate (21) at (1,-2);
    \coordinate (31) at (2,-1);
    \coordinate (42) at (2,1);
    \coordinate (41) at (3,0);
    \draw[ultra thick]
        (14) -- (34) -- (43) -- (41)
        (14) -- (12) -- (21) -- (41)
        (13) -- (23) -- (32) -- (31)
        (24) -- (23)    (32) -- (42)
    ;
    \draw[ultra thick]
        (1122) -- (23)
        (1122) -- (32)
        (2233) -- (12)
        (2233) -- (21)
        (2233) -- (34)
        (2233) -- (43)
        (3344) -- (23)
        (3344) -- (32)
    ;
    \draw[thick, white]
        (1122) -- (23)
        (1122) -- (32)
        (2233) -- (12)
        (2233) -- (21)
        (2233) -- (34)
        (2233) -- (43)
        (3344) -- (23)
        (3344) -- (32)
    ;
    \draw[black,fill=black] 
        (14) circle (3pt)
        (24) circle (3pt)
        (13) circle (3pt)
        (12) circle (3pt)
        (23) circle (3pt)
        (34) circle (3pt)
        (43) circle (3pt)
        (32) circle (3pt)
        (21) circle (3pt)
        (31) circle (3pt)
        (42) circle (3pt)
        (41) circle (3pt)
    ;
    \draw[black, thick,fill=white] 
        (3344) circle (3pt)
        (2233) circle (3pt)
        (1122) circle (3pt)
    ;
    \node at (14) [left=4pt] {$\varpi$};
    \node at (14) [left=3pt, above] {$1\bar4$};
    \node at (24) [left=3pt, above] {$2\bar4$};
    \node at (13) [left=3pt, below] {$1\bar3$};
    \node at (34) [left=3pt, above] {$3\bar4$};
    \node at (23) [left=2] {$2\bar3$};
    \node at (12) [left=3pt, below] {$1\bar2$};
    \node at (1122) [below] {$a$};
    \node at (2233) [above=2pt] {$b$};
    \node at (3344) [above] {$c$};
    \node at (21) [right=3pt, below] {$2\bar1$};
    \node at (32) [right=2pt] {$3\bar2$};
    \node at (43) [right=3pt, above] {$4\bar3$};
    \node at (31) [right=3pt, below] {$3\bar1$};
    \node at (42) [right=3pt, above] {$4\bar2$};
    \node at (41) [right=3pt, above] {$4\bar1$};
    \node at (-2.5,.5) [below=2pt, right=-2pt] {\tiny$1$};
    \node at (-2.5,-.5) [above=2pt, right=-2pt] {\tiny$3$};
    \node at (-1.5,1.5) [below=2pt, right=-2pt] {\tiny$2$};
    \node at (-1.5,.5) [above=2pt, right=-2pt] {\tiny$3$};
    \node at (-1.5,-.5) [below=2pt, right=-2pt] {\tiny$1$};
    \node at (-1.5,-1.5) [above=2pt, right=-2pt] {\tiny$2$};
    \node at (-.5,2) [above=-2pt] {\tiny$3$};
    \node at (-.5,2) [right=2pt, below=0pt] {\tiny$2$};
    \node at (-.5,0) [right=2pt, above=0pt] {\tiny$3$};
    \node at (-.5,0) [below=-2pt] {\tiny$2$};
    \node at (0,0) [left=3pt, below=6pt] {\tiny$1$};
    \node at (0,0) [right=3pt, below=6pt] {\tiny$1$};
    \node at (-.5,-2) [right=2pt, above=0pt] {\tiny$2$};
    \node at (-.5,-2) [below=-2pt] {\tiny$1$};
    \node at (.5,2) [above=-2pt] {\tiny$3$};
    \node at (.5,2) [left=2pt, below=0pt] {\tiny$2$};
    \node at (.5,0) [left=2pt, above=0pt] {\tiny$3$};
    \node at (.5,0) [below=-2pt] {\tiny$2$};
    \node at (.5,-2) [left=2pt, above=0pt] {\tiny$2$};
    \node at (.5,-2) [below=-2pt] {\tiny$1$};
    \node at (1.5,1.5) [below=2pt, left=-2pt] {\tiny$2$};
    \node at (1.5,.5) [above=2pt, left=-2pt] {\tiny$3$};
    \node at (1.5,-.5) [below=2pt, left=-2pt] {\tiny$1$};
    \node at (1.5,-1.5) [above=2pt, left=-2pt] {\tiny$2$};
    \node at (2.5,.5) [below=2pt, left=-2pt] {\tiny$1$};
    \node at (2.5,-.5) [above=2pt, left=-2pt] {\tiny$3$};
\end{tikzpicture}
\]

To obtain the associated Chevalley polytope, we need to fix a reduced expression for $\wP$; to illustrate how well the Chevalley polytope behaves, we will consider two different reduced expressions: $\bs_1=s_3s_2s_1s_2s_3$ and $\bs_2=s_3s_1s_2s_1s_3$. 
The resulting valuations $\nu_1=\nu_{X,\varpi,\bs_1}$ and $\nu_2=\nu_{X,\varpi,\bs_2}$ of the Chevalley bases for the corresponding representations are drawn below; we write $i_1i_2i_3i_4i_5=(i_1,i_2,i_3,i_4,i_5)$ for brevity (e.g.~$00111=(0,0,1,1,1)$), and we have indicated the path taken by the monomial $\boldf$ from the highest weight vector to each weight vector in red:
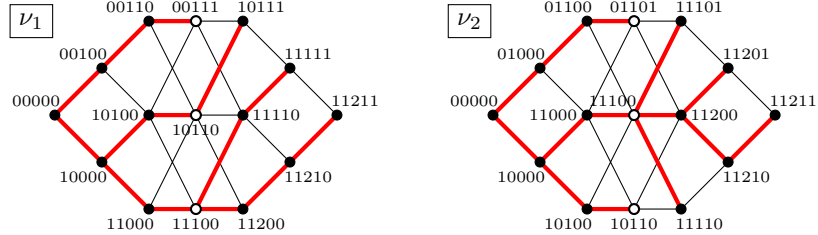
\[
\begin{tikzpicture}[scale=0.625,baseline=0em]
    \coordinate (14) at (-3,0);
    \coordinate (24) at (-2,1);
    \coordinate (13) at (-2,-1);
    \coordinate (12) at (-1,-2);
    \coordinate (23) at (-1,0);
    \coordinate (34) at (-1,2);
    \coordinate (3344) at (0,2);
    \coordinate (2233) at (0,0);
    \coordinate (1122) at (0,-2);
    \coordinate (43) at (1,2);
    \coordinate (32) at (1,0);
    \coordinate (21) at (1,-2);
    \coordinate (31) at (2,-1);
    \coordinate (42) at (2,1);
    \coordinate (41) at (3,0);
    \draw
        (14) -- (34) -- (43) -- (41)
        (14) -- (12) -- (21) -- (41)
        (13) -- (23) -- (32) -- (31)
        (24) -- (23)    (32) -- (42)
    ;
    \draw
        (1122) -- (23)
        (1122) -- (32)
        (2233) -- (12)
        (2233) -- (21)
        (2233) -- (34)
        (2233) -- (43)
        (3344) -- (23)
        (3344) -- (32)
    ;
    \draw[ultra thick, red]
        (14) -- (12) -- (21) -- (41)
        (14) -- (34) -- (3344)
        (13) -- (23) -- (2233)
        (32) -- (42)
    ;
    \draw[ultra thick, red]
        (2233) -- (43)
        (1122) -- (32)
    ;
    \draw[black,fill=black] 
        (14) circle (3pt)
        (24) circle (3pt)
        (13) circle (3pt)
        (12) circle (3pt)
        (23) circle (3pt)
        (34) circle (3pt)
        (43) circle (3pt)
        (32) circle (3pt)
        (21) circle (3pt)
        (31) circle (3pt)
        (42) circle (3pt)
        (41) circle (3pt)
    ;
    \draw[black, thick,fill=white] 
        (3344) circle (3pt)
        (2233) circle (3pt)
        (1122) circle (3pt)
    ;
    \node[draw] at (-3.5,2) {$\nu_1$};
    \node at (14) [left=7pt, above] {\tiny$00000$};
    \node at (24) [left=7pt, above] {\tiny$00100$};
    \node at (13) [left=7pt, below] {\tiny$10000$};
    \node at (34) [left=7pt, above] {\tiny$00110$};
    \node at (23) [left] {\tiny$10100$};
    \node at (12) [left=7pt, below] {\tiny$11000$};
    \node at (1122) [below] {\tiny$11100$};
    \node at (2233) [below] {\tiny$10110$};
    \node at (3344) [above] {\tiny$00111$};
    \node at (21) [right=7pt, below] {\tiny$11200$};
    \node at (32) [right] {\tiny$11110$};
    \node at (43) [right=7pt, above] {\tiny$10111$};
    \node at (31) [right=7pt, below] {\tiny$11210$};
    \node at (42) [right=7pt, above] {\tiny$11111$};
    \node at (41) [right=7pt, above] {\tiny$11211$};
\end{tikzpicture} 
\qquad 
\begin{tikzpicture}[scale=0.625,baseline=0em]
    \coordinate (14) at (-3,0);
    \coordinate (24) at (-2,1);
    \coordinate (13) at (-2,-1);
    \coordinate (12) at (-1,-2);
    \coordinate (23) at (-1,0);
    \coordinate (34) at (-1,2);
    \coordinate (3344) at (0,2);
    \coordinate (2233) at (0,0);
    \coordinate (1122) at (0,-2);
    \coordinate (43) at (1,2);
    \coordinate (32) at (1,0);
    \coordinate (21) at (1,-2);
    \coordinate (31) at (2,-1);
    \coordinate (42) at (2,1);
    \coordinate (41) at (3,0);
    \draw
        (14) -- (34) -- (43) -- (41)
        (14) -- (12) -- (21) -- (41)
        (13) -- (23) -- (32) -- (31)
        (24) -- (23)    (32) -- (42)
    ;
    \draw
        (1122) -- (23)
        (1122) -- (32)
        (2233) -- (12)
        (2233) -- (21)
        (2233) -- (34)
        (2233) -- (43)
        (3344) -- (23)
        (3344) -- (32)
    ;
    \draw[ultra thick, red]
        (14) -- (13) -- (23) -- (32) -- (31) -- (41)
        (14) -- (34) -- (3344)
        (13) -- (12) -- (1122)
        (32) -- (42)
    ;
    \draw[ultra thick, red]
        (2233) -- (43)
        (2233) -- (21)
    ;
    \draw[black,fill=black] 
        (14) circle (3pt)
        (24) circle (3pt)
        (13) circle (3pt)
        (12) circle (3pt)
        (23) circle (3pt)
        (34) circle (3pt)
        (43) circle (3pt)
        (32) circle (3pt)
        (21) circle (3pt)
        (31) circle (3pt)
        (42) circle (3pt)
        (41) circle (3pt)
    ;
    \draw[black, thick,fill=white] 
        (3344) circle (3pt)
        (2233) circle (3pt)
        (1122) circle (3pt)
    ;
    \node[draw] at (-3.5,2) {$\nu_2$};
    \node at (14) [left=7pt, above] {\tiny$00000$};
    \node at (24) [left=7pt, above] {\tiny$01000$};
    \node at (13) [left=7pt, below] {\tiny$10000$};
    \node at (34) [left=7pt, above] {\tiny$01100$};
    \node at (23) [left] {\tiny$11000$};
    \node at (12) [left=7pt, below] {\tiny$10100$};
    \node at (1122) [below] {\tiny$10110$};
    \node at (2233) [left=8pt, above] {\tiny$11100$};
    \node at (3344) [above] {\tiny$01101$};
    \node at (21) [right=7pt, below] {\tiny$11110$};
    \node at (32) [right] {\tiny$11200$};
    \node at (43) [right=7pt, above] {\tiny$11101$};
    \node at (31) [right=7pt, below] {\tiny$11210$};
    \node at (42) [right=7pt, above] {\tiny$11201$};
    \node at (41) [right=7pt, above] {\tiny$11211$};
\end{tikzpicture}
\]
We confirm using Polymake that the volumes of both resulting polytopes $\cP_{\X,\varpi,\bs_1}$ and $\cP_{\X,\varpi,\bs_2}$ agree with the degree of $\X\hookrightarrow\PS\bigl(V(\varpi)\bigr)$, namely they are all $20$. Hence, both polytopes are Newton-Okounkov bodies for $\X$ with respect to the corresponding valuations, and thus Conjecture \ref{conj:nobody} holds in this case. We furthermore checked in Polymake that these polytopes possess the integer decomposition property, which further verifies Conjecture \ref{conj:khovanskii}.

Finally, to verify Conjecture \ref{conj:decomposition}, we embed the minuscule Chevalley polytopes $\cP_1$ and $\cP_3$ according to the reduced expressions $\bs_1$ and $\bs_2$ and take the Minkowski sums of the images. For $\bs_1=s_3s_2s_1s_2s_3$ we find the embedded polytopes (using the same notation as in the above diagrams) as follows, where the entries in bold represent the embedding:
\[
\iota_1(\cP_1) = \conv\{00{\bf000},00{\bf100},00{\bf110},00{\bf111}\};\quad
\iota_1(\cP_3) = \conv\{{\bf000}00,{\bf100}00,{\bf110}00,{\bf111}00\}.
\]
Taking the Minkowski sum, we find:
\[
\iota_1(\cP_1)+\iota_1(\cP_3) = 
\conv\left\{\!\!\!\begin{array}{l}
    00000,10000,11000, 00100,10100,11200, 00110, \\
    10110,11110,11210, 00111,10111,11111,11211
\end{array}\!\!\!\right\} = \cP_{X,\varpi,\bs_1}
\]
confirming the conjecture.

On the other hand, considering $\bs_2=s_3s_1s_2s_1s_3$, we find the embedded polytopes
\[
\iota_2(\cP_1) = \conv\{0{\bf00}0{\bf0},0{\bf10}0{\bf0},0{\bf11}0{\bf0},0{\bf11}0{\bf1}\};\quad
\iota_2(\cP_3) = \conv\{{\bf0}0{\bf00}0,{\bf1}0{\bf00}0,{\bf1}0{\bf10}0,{\bf1}0{\bf11}0\}
\]
and hence the Minkowski sum:
\ali{
\iota_2(\cP_1)+\iota_2(\cP_3) = 
\conv\left\{\!\!\!\begin{array}{l}
00000,01000,01100,01101,
10000,11000,11101,\\
10100,11200,11201,
10110,11110,11210,11211
\end{array}\!\!\!\right\} = \cP_{\X,\varpi,\bs_2}
}
again confirming Conjecture \ref{conj:decomposition}. We will keep the same conventions for the remainder of the appendix.

\addtocontents{toc}{\SkipTocEntry}
\subsection{A maximal orthogonal Grassmannian}
The second example we present is the maximal orthogonal Grassmannian $\X=\OG(2,5)=\Spin_5/\P_2$, where $\Spin_5$ is of type $\LGB_2$, with the embedding into $\bP(V_\varpi)$ for $\varpi=3\fwt[2]$. In this case, we note that $V_\varpi=\Sym^3(V_{\fwt[2]})$, where $V_{\fwt[2]}$ is a minuscule representation of dimension $4$, and that $\bs=s_2s_1s_2$ is the unique reduced expression for $\wP$.

In this case, we will compare the Chevalley polytope $\cP_{\X,\varpi,\bs}$ with the minuscule Chevalley polytope $\cP_m = \cP_{\X,\fwt[2]}$. We already computed $\cP_m$ in Section \ref{sec:string}, and we recall from there that $\minposet = \raisebox{-2pt}{\Yboxdim{6pt}\tiny\young(2,12)}$ and hence we $\cP_m$ is the convex hull of $(0,0,0)$, $(1,0,0)$, $(1,1,0)$ and $(1,1,1)$.

Letting $(1,2,3,4)$ denote the weight basis of the minuscule fundamental representation $V_{\fwt[2]}$ ordered from highest to lowest weight (so we have $\Chf_2:1\mapsto 2,3\mapsto4$ and $\Chf_1:2\mapsto3$),
and writing $ijk=i\odot j\odot k$ for the symmetric tensor, we remark that the actions of monomials $\boldf$ in the Chevalley generators $\Chf_1,\Chf_2$ yield a spanning yet of $V_\varpi$ which is not a basis. Below, we draw a subset $\{\boldf\cdot v_\varpi\}$ which forms a (Chevalley) basis and whose valuations exhaust all possible valuations of degree $1$ elements. Note that we also only drew a subset of the actions of the Chevalley generators $\Chf_i$ and suppressed the labels of the corresponding edges as these can easily deduced from comparing the basis vectors.
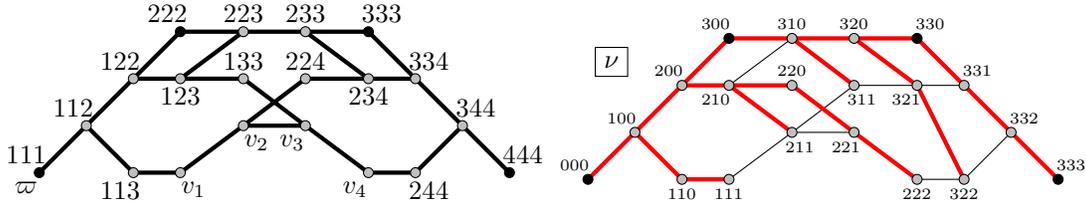
\[
\begin{tikzpicture}[scale=.625]
    \coordinate (111) at (-5,-1);
    \coordinate (112) at (-4,0);
    \coordinate (113) at (-3,-1);
    \coordinate (122) at (-3,1);
    \coordinate (114) at (-2,-1);
    \coordinate (222) at (-2,2);
    \coordinate (123-114) at (-2,1);
    \coordinate (133) at (-.66,1);
    \coordinate (223-124) at (-.66,2);
    \coordinate (124) at (-.66,0);
    \coordinate (233-134) at (.66,2);
    \coordinate (224) at (.66,1);
    \coordinate (134) at (.66,0);
    \coordinate (144) at (2,-1);
    \coordinate (333) at (2,2);
    \coordinate (234-144) at (2,1);
    \coordinate (244) at (3,-1);
    \coordinate (334) at (3,1);
    \coordinate (344) at (4,0);
    \coordinate (444) at (5,-1);
    \draw[ultra thick]
        (111) -- (112) -- (113) -- (114) -- (124) -- (134) -- (144) -- (244) -- (344) -- (444)
        (112) -- (122) -- (123-114) -- (133) -- (134)    (124) -- (224) -- (234-144) -- (334) -- (344)
        (122) -- (222) -- (223-124) -- (233-134) -- (333) -- (334)
        (123-114) -- (223-124) (233-134) -- (234-144)
    ;
    \draw[black, fill=black]
        (111) circle (3pt)
        (222) circle (3pt)
        (333) circle (3pt)
        (444) circle (3pt)
    ;
    \draw[black, fill=lightgray]
        (112) circle (3pt)
        (113) circle (3pt)
        (122) circle (3pt)
        (114) circle (3pt)
        (133) circle (3pt)
        (224) circle (3pt)
        (144) circle (3pt)
        (244) circle (3pt)
        (334) circle (3pt)
        (344) circle (3pt)
    ;
    \draw[black, fill=lightgray]
        (123-114) circle (3pt)
        (124) circle (3pt)
        (223-124) circle (3pt)
        (134) circle (3pt)
        (233-134) circle (3pt)
        (234-144) circle (3pt)
    ;
    \node at (111) [left=5pt, below=0pt] {$\varpi$};
    \node at (111) [left=5pt, above] {$111$};
    \node at (112) [left=5pt, above] {$112$};
    \node at (113) [left=5pt, below] {$113$};
    \node at (122) [left=5pt, above] {$122$};
    \node at (114) [right=5pt, below] {$v_1$};
    \node at (222) [left=5pt, above] {$222$};
    \node at (123-114) [below] {$123$};
    \node at (124) [right=5pt, below] {$v_2$};
    \node at (133) [above] {$133$};
    \node at (223-124) [above] {$223$};
    \node at (134) [left=5pt, below] {$v_3$};
    \node at (233-134) [above] {$233$};
    \node at (224) [above] {$224$};
    \node at (144) [left=5pt, below] {$v_4$};
    \node at (333) [right=5pt, above] {$333$};
    \node at (234-144) [below] {$234$};
    \node at (244) [right=5pt, below] {$244$};
    \node at (334) [right=5pt, above] {$334$};
    \node at (344) [right=5pt, above] {$344$};
    \node at (444) [right=5pt, above] {$444$};
\end{tikzpicture}
\begin{tikzpicture}[scale=.625]
    \coordinate (111) at (-5,-1);
    \coordinate (112) at (-4,0);
    \coordinate (113) at (-3,-1);
    \coordinate (122) at (-3,1);
    \coordinate (114) at (-2,-1);
    \coordinate (222) at (-2,2);
    \coordinate (123-114) at (-2,1);
    \coordinate (133) at (-.66,1);
    \coordinate (223-124) at (-.66,2);
    \coordinate (124) at (-.66,0);
    \coordinate (233-134) at (.66,2);
    \coordinate (224) at (.66,1);
    \coordinate (134) at (.66,0);
    \coordinate (144) at (2,-1);
    \coordinate (333) at (2,2);
    \coordinate (234-144) at (2,1);
    \coordinate (244) at (3,-1);
    \coordinate (334) at (3,1);
    \coordinate (344) at (4,0);
    \coordinate (444) at (5,-1);
    \draw
        (111) -- (112) -- (113) -- (114) -- (124) -- (134) -- (144) -- (244) -- (344) -- (444)
        (112) -- (122) -- (123-114) -- (133) -- (134)    (124) -- (224) -- (234-144) -- (334) -- (344)
        (122) -- (222) -- (223-124) -- (233-134) -- (333) -- (334)
        (123-114) -- (223-124) (233-134) -- (234-144)
    ;
    \draw[ultra thick, red]
        (111) -- (222) -- (333) -- (444)
        (112) -- (113) -- (114)
        (122) -- (123-114) -- (133) -- (134) -- (144)
        (123-114) -- (124)
        (223-124) -- (224)
        (233-134) -- (234-144) -- (244)
    ;
    \draw[black, fill=black]
        (111) circle (3pt)
        (222) circle (3pt)
        (333) circle (3pt)
        (444) circle (3pt)
    ;
    \draw[black, fill=lightgray]
        (112) circle (3pt)
        (113) circle (3pt)
        (122) circle (3pt)
        (114) circle (3pt)
        (133) circle (3pt)
        (224) circle (3pt)
        (144) circle (3pt)
        (244) circle (3pt)
        (334) circle (3pt)
        (344) circle (3pt)
    ;
    \draw[black, fill=lightgray]
        (123-114) circle (3pt)
        (124) circle (3pt)
        (223-124) circle (3pt)
        (134) circle (3pt)
        (233-134) circle (3pt)
        (234-144) circle (3pt)
    ;
    \node[draw] at (-4.5,1.5) {$\nu$};
    \node at (111) [left=5pt, above] {\tiny$000$};
    \node at (112) [left=5pt, above] {\tiny$100$};
    \node at (113) [below] {\tiny$110$};
    \node at (122) [left=5pt, above] {\tiny$200$};
    \node at (114) [below] {\tiny$111$};
    \node at (222) [left=5pt, above] {\tiny$300$};
    \node at (123-114) [left=5pt, below] {\tiny$210$};
    \node at (124) [right=3pt, below] {\tiny$211$};
    \node at (133) [above] {\tiny$220$};
    \node at (223-124) [above] {\tiny$310$};
    \node at (134) [left=3pt, below] {\tiny$221$};
    \node at (233-134) [above] {\tiny$320$};
    \node at (224) [right=3pt, below] {\tiny$311$};
    \node at (144) [below] {\tiny$222$};
    \node at (333) [right=5pt, above] {\tiny$330$};
    \node at (234-144) [left=4pt, below] {\tiny$321$};
    \node at (244) [below] {\tiny$322$};
    \node at (334) [right=5pt, above] {\tiny$331$};
    \node at (344) [right=5pt, above] {\tiny$332$};
    \node at (444) [right=5pt, above] {\tiny$333$};
\end{tikzpicture}
\]
where
\[
    v_1 = 2\cdot 123 + 114, \quad v_2 = 223 + 124, \quad 
    v_3 = 223 + 2\cdot 134, \qand v_4 = 144 + 2\cdot 234.
\]
We find that
\[
\cP_{\X,\varpi,\bs} = \conv\{000,300,330,333\} = 3\cdot\cP_m,
\]
showing that Conjectures \ref{conj:nobody} and \ref{conj:decomposition} hold, since $\vol(\cP_{\X,\varpi,\bs})=27=\deg(\X)$. Furthermore, the lattice points of $\cP_{\X,\varpi,\bs}$ form a Hilbert basis for the cone over $\cP_{\X,\varpi,\bs}$, verifying Conjecture \ref{conj:khovanskii}. 

We note that special care must be taken when computing the valuations in this example. Consider the vectors resulting from the following actions:
\begin{align*}
    a &= \tfrac1{3!}\Chf_1\Chf_2^3\cdot 111 = 223\\
    b &= \tfrac12\Chf_2\Chf_1\Chf_2^2\cdot 111 = 223 + 124\\
    c &= \tfrac12\Chf_2^2\Chf_1\Chf_2 \cdot 111 = 223 + 2\times 124.
\end{align*} 
Since $c\neq 0$, it may seem like the valuation of the corresponding coordinate function $p_c$ should be $112$, corresponding to the monomial $a_1a_2a_3^2$. However, these three vectors span a two-dimensional weight space and in particular satisfy the relation, $-a + 2b - c=0$. Hence, no matter which basis we select, the restriction $p_c|_{\Xo}$ must contain one of the terms (up to scaling) $a_1^3a_2$ or $a_1^2a_2a_3$ corresponding to $a$ or $b$ above, and hence its valuation will be $310$ or $211$, which is already included in the valuations listed above.  

\addtocontents{toc}{\SkipTocEntry}
\subsection{The coadjoint \texorpdfstring{$\LGF_4$}{F4}-Grassmannian}
The third example we present is the coadjoint $\LGF_4$-Grassmannian $\LGF_4^\SC/P_4$, with the minimal embedding into $\bP(V_{\fwt[4]})$. Here only Conjectures \ref{conj:nobody} and \ref{conj:khovanskii} apply, and we now demonstrate that both hold.

Since the fundamental weight $\fwt[4]$ is a \emph{coadjoint weight}, i.e.~$\dfwt[4]$ is adjoint in $\LGF_4^\vee$, the nonzero weight vectors form a single Weyl orbit, just as in the case of the adjoint flag variety $\Fl(1,3;4)$ of type $\LGA_3$. Hence, we have a natural weight basis outside the zero weight space, as well as a natural choice of basis for zero-weight space by $\Chf_3(v_{\sr_3})$ and $\Chf_3(v_{\sr_3})$ for $v_{\sr_i}$ the basis vectors of weight $\sr_i$. Again, we have one-way edges (drawn as ``hollow lines'') from the zero weight spaces to the nonzero-weight spaces.
We now draw the actions of the Chevalley generators on the highest weight vector with weight $\fwt[4]$ below:
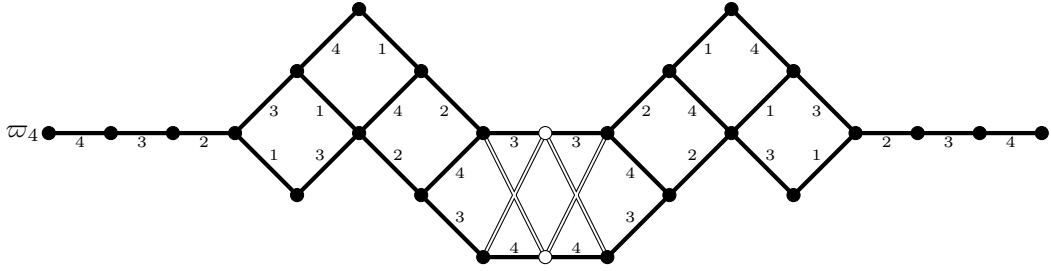
\[
\begin{tikzpicture}[scale=.825]
    \coordinate (1232) at (-8,0);
    \coordinate (1231) at (-7,0);
    \coordinate (1221) at (-6,0);
    \coordinate (1121) at (-5,0);
    \coordinate (1111) at (-4,1);
    \coordinate (0121) at (-4,-1);
    \coordinate (1110) at (-3,2);
    \coordinate (0111) at (-3,0);
    \coordinate (0110) at (-2,1);
    \coordinate (0011) at (-2,-1);
    \coordinate (0010) at (-1,0);
    \coordinate (0001) at (-1,-2);
    \coordinate (h3) at (0,0);
    \coordinate (h4) at (0,-2);
    \coordinate (-0010) at (1,0);
    \coordinate (-0001) at (1,-2);
    \coordinate (-0110) at (2,1);
    \coordinate (-0011) at (2,-1);
    \coordinate (-1110) at (3,2);
    \coordinate (-0111) at (3,0);
    \coordinate (-0121) at (4,-1);
    \coordinate (-1111) at (4,1);
    \coordinate (-1121) at (5,0);
    \coordinate (-1221) at (6,0);
    \coordinate (-1231) at (7,0);
    \coordinate (-1232) at (8,0);
    \draw[ultra thick]
        (1232) -- (1121) -- (1110) -- (0010) -- (-0010) -- (-1110) -- (-1121) -- (-1232)
        (1121) -- (0121) -- (0111) -- (0001) -- (-0001) -- (-0111) -- (-0121) -- (-1121)
        (1111) -- (0111) -- (0110)    (-0110) -- (-0111) -- (-1111)
        (0011) -- (0010)    (-0010) -- (-0011)
    ;
    \draw[ultra thick]
        (h3) -- (0001)
        (h3) -- (-0001)
        (h4) -- (0010)
        (h4) -- (-0010)
    ;
    \draw[thick, white]
        (h3) -- (0001)
        (h3) -- (-0001)
        (h4) -- (0010)
        (h4) -- (-0010)
    ;
    \draw[black, fill=black]
        (1232) circle (3pt)
        (1231) circle (3pt)
        (1221) circle (3pt)
        (1121) circle (3pt)
        (1111) circle (3pt)
        (0121) circle (3pt)
        (1110) circle (3pt)
        (0111) circle (3pt)
        (0110) circle (3pt)
        (0011) circle (3pt)
        (0010) circle (3pt)
        (0001) circle (3pt)
        (-0010) circle (3pt)
        (-0001) circle (3pt)
        (-0110) circle (3pt)
        (-0011) circle (3pt)
        (-1110) circle (3pt)
        (-0111) circle (3pt)
        (-0121) circle (3pt)
        (-1111) circle (3pt)
        (-1121) circle (3pt)
        (-1221) circle (3pt)
        (-1231) circle (3pt)
        (-1232) circle (3pt)
    ;
    \draw[black, fill=white]
        (h3) circle (3pt)
        (h4) circle (3pt)
    ;
    \node at (1232) [left] {$\fwt[4]$};
    \node at (-7.5,0) [below=-2pt] {\tiny$4$};
    \node at (-6.5,0) [below=-2pt] {\tiny$3$};
    \node at (-5.5,0) [below=-2pt] {\tiny$2$};
    \node at (-4.5,.5) [right=3pt, below=-2pt] {\tiny$3$};
    \node at (-4.5,-.5) [right=3pt, above=-2pt] {\tiny$1$};
    \node at (-3.5,1.5) [right=3pt, below=-2pt] {\tiny$4$};
    \node at (-3.5,.5) [left=3pt, below=-2pt] {\tiny$1$};
    \node at (-3.5,-.5) [left=3pt, above=-2pt] {\tiny$3$};
    \node at (-2.5,1.5) [left=3pt, below=-2pt] {\tiny$1$};
    \node at (-2.5,.5) [right=3pt, below=-2pt] {\tiny$4$};
    \node at (-2.5,-.5) [right=3pt, above=-2pt] {\tiny$2$};
    \node at (-1.5,.5) [left=3pt, below=-2pt] {\tiny$2$};
    \node at (-1.5,-.5) [right=3pt, below=-2pt] {\tiny$4$};
    \node at (-1.5,-1.5) [right=3pt, above=-2pt] {\tiny$3$};
    \node at (-.5,0) [below=-2pt] {\tiny$3$};
    \node at (-.5,-2) [above=-2pt] {\tiny$4$};
    \node at (.5,0) [below=-2pt] {\tiny$3$};
    \node at (.5,-2) [above=-2pt] {\tiny$4$};
    \node at (1.5,.5) [right=3pt, below=-2pt] {\tiny$2$};
    \node at (1.5,-.5) [left=3pt, below=-2pt] {\tiny$4$};
    \node at (1.5,-1.5) [left=3pt, above=-2pt] {\tiny$3$};
    \node at (2.5,1.5) [right=3pt, below=-2pt] {\tiny$1$};
    \node at (2.5,.5) [left=3pt, below=-2pt] {\tiny$4$};
    \node at (2.5,-.5) [left=3pt, above=-2pt] {\tiny$2$};
    \node at (3.5,1.5) [left=3pt, below=-2pt] {\tiny$4$};
    \node at (3.5,.5) [right=3pt, below=-2pt] {\tiny$1$};
    \node at (3.5,-.5) [right=3pt, above=-2pt] {\tiny$3$};
    \node at (4.5,.5) [left=3pt, below=-2pt] {\tiny$3$};
    \node at (4.5,-.5) [left=3pt, above=-2pt] {\tiny$1$};
    \node at (5.5,0) [below=-2pt] {\tiny$2$};
    \node at (6.5,0) [below=-2pt] {\tiny$3$};
    \node at (7.5,0) [below=-2pt] {\tiny$4$};
\end{tikzpicture}
\]

There are two types of reduced expression for $\wP$ up to commutations:
\[
\bs_1 = s_4s_3s_2s_1s_3s_2(s_4s_3s_4)s_2s_3s_1s_2s_3s_4\qquad\textrm{and}\qquad
\bs_1 = s_4s_3s_2s_1s_3s_2(s_3s_4s_3)s_2s_3s_1s_2s_3s_4
\]
and we find the following respective valuations $\nu_i=\nu_{\X,\varpi,\bs_i}$ (where in addition to the usual notation we abbreviate recurring values by powers, e.g.~$0^41^3=0000111$):
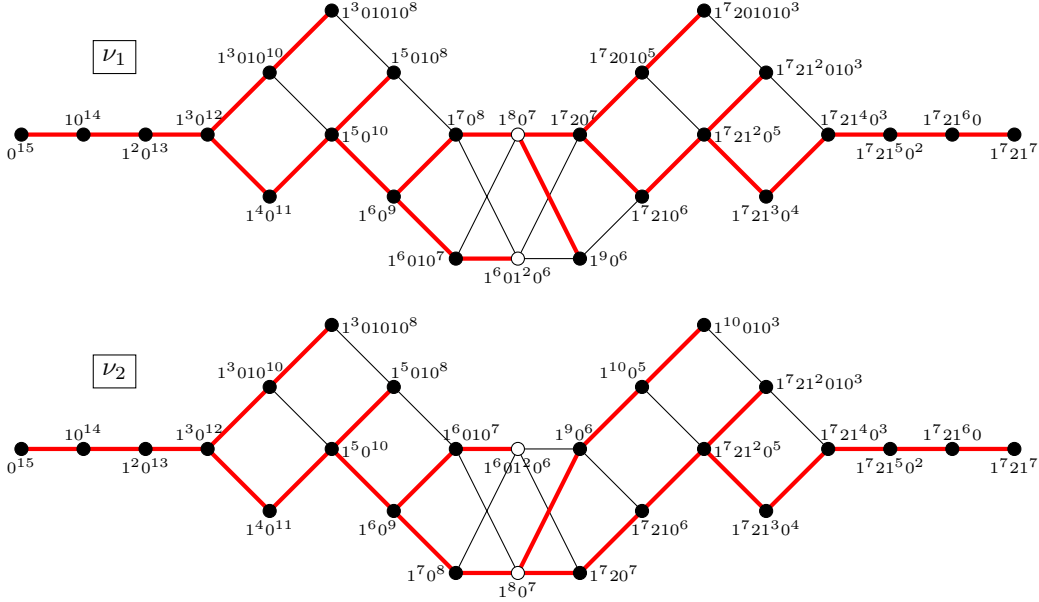
\[
\begin{tikzpicture}[scale=.825]
    \coordinate (1232) at (-8,0);
    \coordinate (1231) at (-7,0);
    \coordinate (1221) at (-6,0);
    \coordinate (1121) at (-5,0);
    \coordinate (1111) at (-4,1);
    \coordinate (0121) at (-4,-1);
    \coordinate (1110) at (-3,2);
    \coordinate (0111) at (-3,0);
    \coordinate (0110) at (-2,1);
    \coordinate (0011) at (-2,-1);
    \coordinate (0010) at (-1,0);
    \coordinate (0001) at (-1,-2);
    \coordinate (h3) at (0,0);
    \coordinate (h4) at (0,-2);
    \coordinate (-0010) at (1,0);
    \coordinate (-0001) at (1,-2);
    \coordinate (-0110) at (2,1);
    \coordinate (-0011) at (2,-1);
    \coordinate (-1110) at (3,2);
    \coordinate (-0111) at (3,0);
    \coordinate (-0121) at (4,-1);
    \coordinate (-1111) at (4,1);
    \coordinate (-1121) at (5,0);
    \coordinate (-1221) at (6,0);
    \coordinate (-1231) at (7,0);
    \coordinate (-1232) at (8,0);
    \draw
        (1232) -- (1121) -- (1110) -- (0010) -- (-0010) -- (-1110) -- (-1121) -- (-1232)
        (1121) -- (0121) -- (0111) -- (0001) -- (-0001) -- (-0111) -- (-0121) -- (-1121)
        (1111) -- (0111) -- (0110)    (-0110) -- (-0111) -- (-1111)
        (0011) -- (0010)    (-0010) -- (-0011)
    ;
    \draw
        (h3) -- (0001)
        (h3) -- (-0001)
        (h4) -- (0010)
        (h4) -- (-0010)
    ;
    \draw[ultra thick, red]
        (1232) -- (1121) -- (0121) -- (0111) -- (0011) -- (0010) -- (-0010) -- (-0011) -- (-0111) -- (-0121) -- (-1121) -- (-1232)
        (1121) -- (1110)
        (0111) -- (0110)
        (0011) -- (0001) -- (h4)
        (-0010) -- (-1110)
        (-1111) -- (-0111)
    ;
    \draw[ultra thick, red]
        (h3) -- (-0001)
    ;
    \draw[black, fill=black]
        (1232) circle (3pt)
        (1231) circle (3pt)
        (1221) circle (3pt)
        (1121) circle (3pt)
        (1111) circle (3pt)
        (0121) circle (3pt)
        (1110) circle (3pt)
        (0111) circle (3pt)
        (0110) circle (3pt)
        (0011) circle (3pt)
        (0010) circle (3pt)
        (0001) circle (3pt)
        (-0010) circle (3pt)
        (-0001) circle (3pt)
        (-0110) circle (3pt)
        (-0011) circle (3pt)
        (-1110) circle (3pt)
        (-0111) circle (3pt)
        (-0121) circle (3pt)
        (-1111) circle (3pt)
        (-1121) circle (3pt)
        (-1221) circle (3pt)
        (-1231) circle (3pt)
        (-1232) circle (3pt)
    ;
    \draw[black, fill=white]
        (h3) circle (3pt)
        (h4) circle (3pt)
    ;
    \node[draw] at (-6.5,1.25) {$\nu_1$};
    \node at (1232) [below] {\tiny$0^{15}$};
    \node at (1231) [above] {\tiny$10^{14}$};
    \node at (1221) [below] {\tiny$1^20^{13}$};
    \node at (1121) [left=3pt, above] {\tiny$1^30^{12}$};
    \node at (1111) [left=7pt, above] {\tiny$1^3010^{10}$};
    \node at (1110) [right] {\tiny$1^301010^8$};
    \node at (0121) [below] {\tiny$1^40^{11}$};
    \node at (0111) [right] {\tiny$1^50^{10}$};
    \node at (0110) [right=10pt, above] {\tiny$1^5010^8$};
    \node at (0011) [left=6pt, below] {\tiny$1^60^9$};
    \node at (0001) [left] {\tiny$1^6010^7$};
    \node at (h4) [below] {\tiny$1^601^20^6$};
    \node at (0010) [right=4pt, above] {\tiny$1^70^8$};
    \node at (h3) [above] {\tiny$1^80^7$};
    \node at (-0001) [right] {\tiny$1^90^6$};
    \node at (-0010) [left=2pt, above] {\tiny$1^720^7$};
    \node at (-0110) [left=7pt, above] {\tiny$1^72010^5$};
    \node at (-1110) [right] {\tiny$1^7201010^3$};
    \node at (-0011) [right=7pt, below] {\tiny$1^7210^6$};
    \node at (-0111) [right] {\tiny$1^721^20^5$};
    \node at (-1111) [above=2pt, right] {\tiny$1^721^2010^3$};
    \node at (-0121) [below] {\tiny$1^721^30^4$};
    \node at (-1121) [right=10pt, above] {\tiny$1^721^40^3$};
    \node at (-1221) [below] {\tiny$1^721^50^2$};
    \node at (-1231) [above] {\tiny$1^721^60$};
    \node at (-1232) [below] {\tiny$1^721^7$};
\end{tikzpicture}
\]
\[
\begin{tikzpicture}[scale=.825]
    \coordinate (1232) at (-8,0);
    \coordinate (1231) at (-7,0);
    \coordinate (1221) at (-6,0);
    \coordinate (1121) at (-5,0);
    \coordinate (1111) at (-4,1);
    \coordinate (0121) at (-4,-1);
    \coordinate (1110) at (-3,2);
    \coordinate (0111) at (-3,0);
    \coordinate (0110) at (-2,1);
    \coordinate (0011) at (-2,-1);
    \coordinate (0010) at (-1,0);
    \coordinate (0001) at (-1,-2);
    \coordinate (h3) at (0,0);
    \coordinate (h4) at (0,-2);
    \coordinate (-0010) at (1,0);
    \coordinate (-0001) at (1,-2);
    \coordinate (-0110) at (2,1);
    \coordinate (-0011) at (2,-1);
    \coordinate (-1110) at (3,2);
    \coordinate (-0111) at (3,0);
    \coordinate (-0121) at (4,-1);
    \coordinate (-1111) at (4,1);
    \coordinate (-1121) at (5,0);
    \coordinate (-1221) at (6,0);
    \coordinate (-1231) at (7,0);
    \coordinate (-1232) at (8,0);
    \draw
        (1232) -- (1121) -- (1110) -- (0010) -- (-0010) -- (-1110) -- (-1121) -- (-1232)
        (1121) -- (0121) -- (0111) -- (0001) -- (-0001) -- (-0111) -- (-0121) -- (-1121)
        (1111) -- (0111) -- (0110)    (-0110) -- (-0111) -- (-1111)
        (0011) -- (0010)    (-0010) -- (-0011)
    ;
    \draw
        (h3) -- (0001)
        (h3) -- (-0001)
        (h4) -- (0010)
        (h4) -- (-0010)
    ;
    \draw[ultra thick, red]
        (1232) -- (1121) -- (0121) -- (0111) -- (0011) -- (0010) -- (h3)
        (-0011) -- (-0111) -- (-0121) -- (-1121) -- (-1232)
        (1121) -- (1110)
        (0111) -- (0110)
        (0011) -- (0001) -- (-0001) -- (-0011)
        (-0010) -- (-1110)
        (-1111) -- (-0111)
    ;
    \draw[ultra thick, red]
        (h4) -- (-0010)
    ;
    \draw[black, fill=black]
        (1232) circle (3pt)
        (1231) circle (3pt)
        (1221) circle (3pt)
        (1121) circle (3pt)
        (1111) circle (3pt)
        (0121) circle (3pt)
        (1110) circle (3pt)
        (0111) circle (3pt)
        (0110) circle (3pt)
        (0011) circle (3pt)
        (0010) circle (3pt)
        (0001) circle (3pt)
        (-0010) circle (3pt)
        (-0001) circle (3pt)
        (-0110) circle (3pt)
        (-0011) circle (3pt)
        (-1110) circle (3pt)
        (-0111) circle (3pt)
        (-0121) circle (3pt)
        (-1111) circle (3pt)
        (-1121) circle (3pt)
        (-1221) circle (3pt)
        (-1231) circle (3pt)
        (-1232) circle (3pt)
    ;
    \draw[black, fill=white]
        (h3) circle (3pt)
        (h4) circle (3pt)
    ;
    \node[draw] at (-6.5,1.25) {$\nu_2$};
    \node at (1232) [below] {\tiny$0^{15}$};
    \node at (1231) [above] {\tiny$10^{14}$};
    \node at (1221) [below] {\tiny$1^20^{13}$};
    \node at (1121) [left=3pt, above] {\tiny$1^30^{12}$};
    \node at (1111) [left=7pt, above] {\tiny$1^3010^{10}$};
    \node at (1110) [right] {\tiny$1^301010^8$};
    \node at (0121) [below] {\tiny$1^40^{11}$};
    \node at (0111) [right] {\tiny$1^50^{10}$};
    \node at (0110) [right=10pt, above] {\tiny$1^5010^8$};
    \node at (0011) [left=6pt, below] {\tiny$1^60^9$};
    \node at (0001) [left] {\tiny$1^70^8$};
    \node at (h4) [below] {\tiny$1^80^7$};
    \node at (0010) [right=6pt,above] {\tiny$1^6010^7$};
    \node at (h3) [below] {\tiny$1^601^20^6$};
    \node at (-0001) [right] {\tiny$1^720^7$};
    \node at (-0010) [left=2pt,above] {\tiny$1^90^6$};
    \node at (-0110) [left=7pt, above] {\tiny$1^{10}0^5$};
    \node at (-1110) [right] {\tiny$1^{10}010^3$};
    \node at (-0011) [right=7pt, below] {\tiny$1^7210^6$};
    \node at (-0111) [right] {\tiny$1^721^20^5$};
    \node at (-1111) [above=2pt, right] {\tiny$1^721^2010^3$};
    \node at (-0121) [below] {\tiny$1^721^30^4$};
    \node at (-1121) [right=10pt, above] {\tiny$1^721^40^3$};
    \node at (-1221) [below] {\tiny$1^721^50^2$};
    \node at (-1231) [above] {\tiny$1^721^60$};
    \node at (-1232) [below] {\tiny$1^721^7$};
\end{tikzpicture}
\]

We verify that the first polytope $\cP_{\X,\dfwt[4],\bs_1}$ has volume $78=\deg(\X)$ and its lattice points form a Hilbert basis, so both Conjectures \ref{conj:nobody} and \ref{conj:khovanskii} are satisfied. However, $\cP_{\X,\dfwt[4],\bs_2}$ has normalized volume $63$, so is not a Newton-Okounkov body for $\X$.

\addtocontents{toc}{\SkipTocEntry}
\subsection{The adjoint \texorpdfstring{$\LGG_2$}{G2}-Grassmannian}
The fourth example we present is the $\LGG_2$-adjoint Grassmannian $\LGG_2^\SC/P_2$, with the minimal embedding into $\bP(V_{\fwt[2]})$. As the fundamental weight $\fwt[2]$ is adjoint but \emph{not} coadjoint, we find that the nonzero weights are the roots and the corresponding spaces are all one-dimensional, but they fall into \emph{two} Weyl orbits: the extremal orbit consists of all long roots; and the non-extremal orbit consisting of all short roots. A basis for the nonzero-weight spaces is therefore determined up to scaling, and we find a natural basis for the zero-weight spaces as usual by considering $\Chf_i(v_{\sr_i})$ for $i\in\{1,2\}$. We find the following Chevalley actions, with the ``hollow'' one-directional edges leaving the zero-weight spaces:
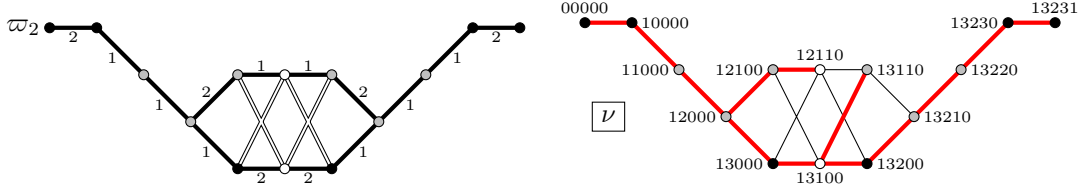
\[
\begin{tikzpicture}[scale = .625]
    \coordinate (32) at (-5,3);
    \coordinate (31) at (-4,3);
    \coordinate (21) at (-3,2);
    \coordinate (11) at (-2,1);
    \coordinate (10) at (-1,2);
    \coordinate (01) at (-1,0);
    \coordinate (h1) at (0,2);
    \coordinate (h2) at (0,0);
    \coordinate (-10) at (1,2);
    \coordinate (-01) at (1,0);
    \coordinate (-11) at (2,1);
    \coordinate (-21) at (3,2);
    \coordinate (-31) at (4,3);
    \coordinate (-32) at (5,3);
    \draw[ultra thick]
        (32) -- (31) -- (01) -- (-01) -- (-31) -- (-32)
        (11) -- (10) -- (-10) -- (-11)
    ;
    \draw[ultra thick]
        (h1) -- (01)
        (h1) -- (-01)
        (h2) -- (10)
        (h2) -- (-10)
    ;
    \draw[thick, white]
        (h1) -- (01)
        (h1) -- (-01)
        (h2) -- (10)
        (h2) -- (-10)
    ;
    \draw[black, fill=black]
        (32) circle (3pt)
        (31) circle (3pt)
        (01) circle (3pt)
        (-01) circle (3pt)
        (-31) circle (3pt)
        (-32) circle (3pt)
    ;
    \draw[black, fill=lightgray]
        (21) circle (3pt)
        (11) circle (3pt)
        (10) circle (3pt)
        (-10) circle (3pt)
        (-11) circle (3pt)
        (-21) circle (3pt)
    ;
    \draw[black, fill=white]
        (h1) circle (3pt)
        (h2) circle (3pt)
    ; 
    \node at (32) [left] {$\fwt[2]$};
    \node at (-4.5,3) [below=-2pt] {\tiny$2$};
    \node at (-3.5,2.5) [left=3pt, below=-3pt] {\tiny$1$};
    \node at (-2.5,1.5) [left=3pt, below=-3pt] {\tiny$1$};
    \node at (-1.5,.5) [left=3pt, below=-3pt] {\tiny$1$};
    \node at (-1.5,1.5) [left=3pt, above=-3pt] {\tiny$2$};
    \node at (-.5,0) [below=-2pt] {\tiny$2$};
    \node at (-.5,2) [above=-2pt] {\tiny$1$};
    \node at (.5,0) [below=-2pt] {\tiny$2$};
    \node at (.5,2) [above=-2pt] {\tiny$1$};
    \node at (1.5,1.5) [right=3pt, above=-3pt] {\tiny$2$};
    \node at (1.5,.5) [right=3pt, below=-3pt] {\tiny$1$};
    \node at (2.5,1.5) [right=3pt, below=-3pt] {\tiny$1$};
    \node at (3.5,2.5) [right=3pt, below=-3pt] {\tiny$1$};
    \node at (4.5,3) [below=-2pt] {\tiny$2$};
\end{tikzpicture}
\quad
\begin{tikzpicture}[scale = .625]
    \coordinate (32) at (-5,3);
    \coordinate (31) at (-4,3);
    \coordinate (21) at (-3,2);
    \coordinate (11) at (-2,1);
    \coordinate (10) at (-1,2);
    \coordinate (01) at (-1,0);
    \coordinate (h1) at (0,2);
    \coordinate (h2) at (0,0);
    \coordinate (-10) at (1,2);
    \coordinate (-01) at (1,0);
    \coordinate (-11) at (2,1);
    \coordinate (-21) at (3,2);
    \coordinate (-31) at (4,3);
    \coordinate (-32) at (5,3);
    \draw
        (32) -- (31) -- (01) -- (-01) -- (-31) -- (-32)
        (11) -- (10) -- (-10) -- (-11)
    ;
    \draw
        (h1) -- (01)
        (h1) -- (-01)
        (h2) -- (10)
        (h2) -- (-10)
    ;
    \draw[ultra thick, red]
        (32) -- (31) -- (01) -- (-01) -- (-31) -- (-32)
        (11) -- (10) -- (h1)
    ;
    \draw[ultra thick, red]
        (h2) -- (-10)
    ;
    \draw[black, fill=black]
        (32) circle (3pt)
        (31) circle (3pt)
        (01) circle (3pt)
        (-01) circle (3pt)
        (-31) circle (3pt)
        (-32) circle (3pt)
    ;
    \draw[black, fill=lightgray]
        (21) circle (3pt)
        (11) circle (3pt)
        (10) circle (3pt)
        (-10) circle (3pt)
        (-11) circle (3pt)
        (-21) circle (3pt)
    ;
    \draw[black, fill=white]
        (h1) circle (3pt)
        (h2) circle (3pt)
    ; 
    \node[draw] at (-4.5,1) {$\nu$};
    \node at (32) [above] {\tiny$00000$};
    \node at (31) [right] {\tiny$10000$};
    \node at (21) [left] {\tiny$11000$};
    \node at (11) [left] {\tiny$12000$};
    \node at (10) [left] {\tiny$12100$};
    \node at (01) [left] {\tiny$13000$};
    \node at (h1) [above] {\tiny$12110$};
    \node at (h2) [below] {\tiny$13100$};
    \node at (-10) [right] {\tiny$13110$};
    \node at (-01) [right] {\tiny$13200$};
    \node at (-11) [right] {\tiny$13210$};
    \node at (-21) [right] {\tiny$13220$};
    \node at (-31) [left] {\tiny$13230$};
    \node at (-32) [above] {\tiny$13231$};
\end{tikzpicture}
\]

There is only one choice of reduced expression $\bs=s_2s_1s_2s_1s_2$ for $\wP$, with the corresponding valuation $\nu=\nu_{\X,\fwt[2],\bs}$ drawn to the right. We find that $\cP_{\X,\fwt[2],\bs}$ has volume $18=\deg(\X)$ and that the lattice points form a Hilbert basis for the cone over $\cP_{\X,\fwt[2],\bs}$ so that both Conjectures \ref{conj:nobody} and \ref{conj:khovanskii} are satisfied. Furthermore, it is interesting to note that $\cP_{\X,\fwt[2],\bs}$ has six vertices, which are obtained from the valuations of the extremal weights. (This is not true in general -- in the previous $\LGF_4$ example as well as the previous $\LGA_3$ example, the non-extremal zero weight spaces contributed vertices to the polytope.)

\addtocontents{toc}{\SkipTocEntry}
\subsection{A coadjoint isotropic Grassmannian}
The final example we present is the homogeneous space $\X=\mathrm{SG}(2,8) = \Sp_8/P_2$, where $\Sp_8$ is of type $\LGC_4$, with the minimal embedding into $\bP(V_{\fwt[2]})$. Here the fundamental weight $\fwt[2]$ is coadjoint as before, and hence the nonzero weights are exactly the short roots and form a single Weyl orbit. A basis for the nonzero-weight spaces is therefore determined up to scaling, and as usual we obtain a natural basis for the zero-weight spaces by considering $\Chf_i(v_{\sr_i})$ for $i\in\{1,2,3\}$; these again allow one-directional edges to the nonzero-weight spaces:
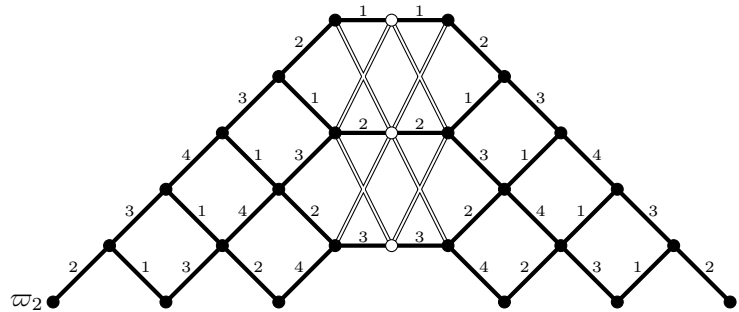
\[
\begin{tikzpicture}[scale=.75]
    \coordinate (1221) at (-6,-1);
    \coordinate (1121) at (-5,0);
    \coordinate (1111) at (-4,1);
    \coordinate (0121) at (-4,-1);
    \coordinate (1110) at (-3,2);
    \coordinate (0111) at (-3,0);
    \coordinate (1100) at (-2,3);
    \coordinate (0110) at (-2,1);
    \coordinate (0011) at (-2,-1);
    \coordinate (1000) at (-1,4);
    \coordinate (0100) at (-1,2);
    \coordinate (0010) at (-1,0);
    \coordinate (h1) at (0,4);
    \coordinate (h2) at (0,2);
    \coordinate (h3) at (0,0);
    \coordinate (-1000) at (1,4);
    \coordinate (-0100) at (1,2);
    \coordinate (-0010) at (1,0);
    \coordinate (-1100) at (2,3);
    \coordinate (-0110) at (2,1);
    \coordinate (-0011) at (2,-1);
    \coordinate (-1110) at (3,2);
    \coordinate (-0111) at (3,0);
    \coordinate (-1111) at (4,1);
    \coordinate (-0121) at (4,-1);
    \coordinate (-1121) at (5,0);
    \coordinate (-1221) at (6,-1);
    \draw[ultra thick]
        (1221) -- (1000) -- (-1000) -- (-1221)
        (1121) -- (0121) -- (0100) -- (-0100) -- (-0121) -- (-1121)
        (1111) -- (0011) -- (0010) -- (-0010) -- (-0011) -- (-1111)
        (1110) -- (0010)    (-0010) -- (-1110)
        (1100) -- (0100)    (-0100) -- (-1100)
        ;
    \draw[ultra thick]
        (h1) -- (0100)
        (h1) -- (-0100)
        (h2) -- (1000)
        (h2) -- (-1000)
        (h2) -- (0010)
        (h2) -- (-0010)
        (h3) -- (0100)
        (h3) -- (-0100)
    ;
    \draw[thick, white]
        (h1) -- (0100)
        (h1) -- (-0100)
        (h2) -- (1000)
        (h2) -- (-1000)
        (h2) -- (0010)
        (h2) -- (-0010)
        (h3) -- (0100)
        (h3) -- (-0100)
    ;
    \draw[black, fill=black]
        (1221) circle (3pt)
        (1121) circle (3pt)
        (1111) circle (3pt)
        (0121) circle (3pt)
        (1110) circle (3pt)
        (0111) circle (3pt)
        (1100) circle (3pt)
        (0110) circle (3pt)
        (0011) circle (3pt)
        (1000) circle (3pt)
        (0100) circle (3pt)
        (0010) circle (3pt)
        (-1000) circle (3pt)
        (-0100) circle (3pt)
        (-0010) circle (3pt)
        (-1100) circle (3pt)
        (-0110) circle (3pt)
        (-0011) circle (3pt)
        (-1110) circle (3pt)
        (-0111) circle (3pt)
        (-1111) circle (3pt)
        (-0121) circle (3pt)
        (-1121) circle (3pt)
        (-1221) circle (3pt)
    ;
    \draw[black, fill=white]
        (h1) circle (3pt)
        (h2) circle (3pt)
        (h3) circle (3pt)
    ;
    \node at (1221) [left] {$\fwt[2]$};
    \node at (-5.5,-.5) [left=3pt, above=-3pt] {\tiny$2$};
    \node at (-4.5,.5) [left=3pt, above=-3pt] {\tiny$3$};
    \node at (-3.5,1.5) [left=3pt, above=-3pt] {\tiny$4$};
    \node at (-2.5,2.5) [left=3pt, above=-3pt] {\tiny$3$};
    \node at (-1.5,3.5) [left=3pt, above=-3pt] {\tiny$2$};
    \node at (-.5,4) [above=-2pt] {\tiny$1$};
    \node at (-4.5,-.5) [right=3pt, above=-3pt] {\tiny$1$};
    \node at (-3.5,.5) [right=3pt, above=-3pt] {\tiny$1$};
    \node at (-2.5,1.5) [right=3pt, above=-3pt] {\tiny$1$};
    \node at (-1.5,2.5) [right=3pt, above=-3pt] {\tiny$1$};
    \node at (-3.5,-.5) [left=3pt, above=-3pt] {\tiny$3$};
    \node at (-2.5,.5) [left=3pt, above=-3pt] {\tiny$4$};
    \node at (-1.5,1.5) [left=3pt, above=-3pt] {\tiny$3$};
    \node at (-.5,2) [above=-2pt] {\tiny$2$};
    \node at (-2.5,-.5) [right=3pt, above=-3pt] {\tiny$2$};
    \node at (-1.5,.5) [right=3pt, above=-3pt] {\tiny$2$};
    \node at (1.5,.5) [left=3pt, above=-3pt] {\tiny$2$};
    \node at (2.5,-.5) [left=3pt, above=-3pt] {\tiny$2$};
    \node at (-1.5,-.5) [left=3pt, above=-3pt] {\tiny$4$};
    \node at (-.5,0) [above=-2pt] {\tiny$3$};
    \node at (.5,0) [above=-2pt] {\tiny$3$};
    \node at (1.5,-.5) [right=3pt, above=-3pt] {\tiny$4$};
    \node at (.5,2) [above=-2pt] {\tiny$2$};
    \node at (1.5,1.5) [right=3pt, above=-3pt] {\tiny$3$};
    \node at (2.5,.5) [right=3pt, above=-3pt] {\tiny$4$};
    \node at (3.5,-.5) [right=3pt, above=-3pt] {\tiny$3$};
    \node at (1.5,2.5) [left=3pt, above=-3pt] {\tiny$1$};
    \node at (2.5,1.5) [left=3pt, above=-3pt] {\tiny$1$};
    \node at (3.5,.5) [left=3pt, above=-3pt] {\tiny$1$};
    \node at (4.5,-.5) [left=3pt, above=-3pt] {\tiny$1$};
    \node at (.5,4) [above=-2pt] {\tiny$1$};
    \node at (1.5,3.5) [right=3pt, above=-3pt] {\tiny$2$};
    \node at (2.5,2.5) [right=3pt, above=-3pt] {\tiny$3$};
    \node at (3.5,1.5) [right=3pt, above=-3pt] {\tiny$4$};
    \node at (4.5,.5) [right=3pt, above=-3pt] {\tiny$3$};
    \node at (5.5,-.5) [right=3pt, above=-3pt] {\tiny$2$};
\end{tikzpicture}
\]
In this case, there are many choices of reduced expression for $\wP$, and we will illustrate three such choices below, namely:
\ali{
\bs_1 &= s_2s_3s_4s_3(s_2s_1s_2)s_3s_4s_3s_2 \\
\bs_2 &= s_2s_3s_4s_3(s_1s_2s_1)s_3s_4s_3s_2 \\
\bs_3 &= s_2(s_1s_3s_2s_4s_3s_4s_2s_3s_1)s_2
}
yielding the following valuations $\nu_i=\nu_{\X,\fwt[2],\bs_i}$ given by:
\[
\begin{tikzpicture}[scale=.75]
    \coordinate (1221) at (-6,-1);
    \coordinate (1121) at (-5,0);
    \coordinate (1111) at (-4,1);
    \coordinate (0121) at (-4,-1);
    \coordinate (1110) at (-3,2);
    \coordinate (0111) at (-3,0);
    \coordinate (1100) at (-2,3);
    \coordinate (0110) at (-2,1);
    \coordinate (0011) at (-2,-1);
    \coordinate (1000) at (-1,4);
    \coordinate (0100) at (-1,2);
    \coordinate (0010) at (-1,0);
    \coordinate (h1) at (0,4);
    \coordinate (h2) at (0,2);
    \coordinate (h3) at (0,0);
    \coordinate (-1000) at (1,4);
    \coordinate (-0100) at (1,2);
    \coordinate (-0010) at (1,0);
    \coordinate (-1100) at (2,3);
    \coordinate (-0110) at (2,1);
    \coordinate (-0011) at (2,-1);
    \coordinate (-1110) at (3,2);
    \coordinate (-0111) at (3,0);
    \coordinate (-1111) at (4,1);
    \coordinate (-0121) at (4,-1);
    \coordinate (-1121) at (5,0);
    \coordinate (-1221) at (6,-1);
    \draw
        (1221) -- (1000) -- (-1000) -- (-1221)
        (1121) -- (0121) -- (0100) -- (-0100) -- (-0121) -- (-1121)
        (1111) -- (0011) -- (0010) -- (-0010) -- (-0011) -- (-1111)
        (1110) -- (0010)    (-0010) -- (-1110)
        (1100) -- (0100)    (-0100) -- (-1100)
    ;
    \draw
        (h1) -- (0100)
        (h1) -- (-0100)
        (h2) -- (1000)
        (h2) -- (-1000)
        (h2) -- (0010)
        (h2) -- (-0010)
        (h3) -- (0100)
        (h3) -- (-0100)
    ;
    \draw[ultra thick, red]
        (1221) -- (1000) -- (h1) -- (-1000) -- (-1221)
        (-0100) -- (-0121)
        (1100) -- (0100) -- (h2)    (-0010) -- (-0011)
        (1110) -- (0010) -- (h3)
        (1111) -- (0011)
        (1121) -- (0121)
    ;
    \draw[ultra thick, red]
        (h1) -- (-0100)
        (h2) -- (-0010)
    ;
    \draw[black, fill=black]
        (1221) circle (3pt)
        (1121) circle (3pt)
        (1111) circle (3pt)
        (0121) circle (3pt)
        (1110) circle (3pt)
        (0111) circle (3pt)
        (1100) circle (3pt)
        (0110) circle (3pt)
        (0011) circle (3pt)
        (1000) circle (3pt)
        (0100) circle (3pt)
        (0010) circle (3pt)
        (-1000) circle (3pt)
        (-0100) circle (3pt)
        (-0010) circle (3pt)
        (-1100) circle (3pt)
        (-0110) circle (3pt)
        (-0011) circle (3pt)
        (-1110) circle (3pt)
        (-0111) circle (3pt)
        (-1111) circle (3pt)
        (-0121) circle (3pt)
        (-1121) circle (3pt)
        (-1221) circle (3pt)
    ;
    \draw[black, fill=white]
        (h1) circle (3pt)
        (h2) circle (3pt)
        (h3) circle (3pt)
    ;
    \node[draw] at (-4,3) {$\nu_1$};
    \node at (1221) [left] {\tiny$0^{11}$};
    \node at (1121) [left] {\tiny$10^{10}$};
    \node at (0121) [left] {\tiny$10^410^5$};
    \node at (1111) [left] {\tiny$1^20^9$};
    \node at (0111) [below=1pt, left=-1pt] {\tiny$1^20^310^5$};
    \node at (0011) [left] {\tiny$1^20^31^20^4$};
    \node at (1110) [left] {\tiny$1^30^8$};
    \node at (0110) [below=1pt, left=-1pt] {\tiny$1^30^210^5$};
    \node at (0010) [below=1pt, left=-1.2pt] {\tiny$1^30^21^20^4$};
    \node at (h3) [below] {\tiny$1^30^21^30^3$};
    \node at (1100) [left] {\tiny$1^40^7$};
    \node at (0100) [below=1pt, left=-1pt] {\tiny$1^4010^5$};
    \node at (h2) [left=1pt, above] {\tiny$1^401^20^4$};
    \node at (-0010) [above=1pt, right=-1pt] {\tiny$1^401^30^3$};
    \node at (-0011) [right=-1pt] {\tiny$1^401^40^2$};
    \node at (1000) [left] {\tiny$1^50^6$};
    \node at (h1) [above] {\tiny$1^60^5$};
    \node at (-0100) [above=1pt, right] {\tiny$1^70^4$};
    \node at (-0110) [above=1pt, right] {\tiny$1^80^3$};
    \node at (-0111) [above=1pt, right] {\tiny$1^90^2$};
    \node at (-0121) [above=1pt, right] {\tiny$1^{10}0$};
    \node at (-1000) [right] {\tiny$1^520^5$};
    \node at (-1100) [right] {\tiny$1^5210^4$};
    \node at (-1110) [right] {\tiny$1^521^20^3$};
    \node at (-1111) [right] {\tiny$1^521^30^2$};
    \node at (-1121) [right] {\tiny$1^521^40$};
    \node at (-1221) [right] {\tiny$1^521^5$};
\end{tikzpicture}
\]
\[
\begin{tikzpicture}[scale=.75]
    \coordinate (1221) at (-6,-1);
    \coordinate (1121) at (-5,0);
    \coordinate (1111) at (-4,1);
    \coordinate (0121) at (-4,-1);
    \coordinate (1110) at (-3,2);
    \coordinate (0111) at (-3,0);
    \coordinate (1100) at (-2,3);
    \coordinate (0110) at (-2,1);
    \coordinate (0011) at (-2,-1);
    \coordinate (1000) at (-1,4);
    \coordinate (0100) at (-1,2);
    \coordinate (0010) at (-1,0);
    \coordinate (h1) at (0,4);
    \coordinate (h2) at (0,2);
    \coordinate (h3) at (0,0);
    \coordinate (-1000) at (1,4);
    \coordinate (-0100) at (1,2);
    \coordinate (-0010) at (1,0);
    \coordinate (-1100) at (2,3);
    \coordinate (-0110) at (2,1);
    \coordinate (-0011) at (2,-1);
    \coordinate (-1110) at (3,2);
    \coordinate (-0111) at (3,0);
    \coordinate (-1111) at (4,1);
    \coordinate (-0121) at (4,-1);
    \coordinate (-1121) at (5,0);
    \coordinate (-1221) at (6,-1);
    \draw
        (1221) -- (1000) -- (-1000) -- (-1221)
        (1121) -- (0121) -- (0100) -- (-0100) -- (-0121) -- (-1121)
        (1111) -- (0011) -- (0010) -- (-0010) -- (-0011) -- (-1111)
        (1110) -- (0010)    (-0010) -- (-1110)
        (1100) -- (0100)    (-0100) -- (-1100)
    ;
    \draw
        (h1) -- (0100)
        (h1) -- (-0100)
        (h2) -- (1000)
        (h2) -- (-1000)
        (h2) -- (0010)
        (h2) -- (-0010)
        (h3) -- (0100)
        (h3) -- (-0100)
    ;
    \draw[ultra thick, red]
        (1221) -- (1100) -- (0100) -- (-0100) -- (-1100) -- (-1221)
        (-0121) -- (-0100)
        (1100) -- (1000) -- (h1)
        (-0010) -- (-0011)
        (1110) -- (0010) -- (h3)
        (1111) -- (0011)
        (1121) -- (0121)
    ;
    \draw[ultra thick, red]
        (h2) -- (-1000)
        (h2) -- (-0010)
    ;
    \draw[black, fill=black]
        (1221) circle (3pt)
        (1121) circle (3pt)
        (1111) circle (3pt)
        (0121) circle (3pt)
        (1110) circle (3pt)
        (0111) circle (3pt)
        (1100) circle (3pt)
        (0110) circle (3pt)
        (0011) circle (3pt)
        (1000) circle (3pt)
        (0100) circle (3pt)
        (0010) circle (3pt)
        (-1000) circle (3pt)
        (-0100) circle (3pt)
        (-0010) circle (3pt)
        (-1100) circle (3pt)
        (-0110) circle (3pt)
        (-0011) circle (3pt)
        (-1110) circle (3pt)
        (-0111) circle (3pt)
        (-1111) circle (3pt)
        (-0121) circle (3pt)
        (-1121) circle (3pt)
        (-1221) circle (3pt)
    ;
    \draw[black, fill=white]
        (h1) circle (3pt)
        (h2) circle (3pt)
        (h3) circle (3pt)
    ;
    \node[draw] at (-4,3) {$\nu_2$};
    \node at (1221) [left] {\tiny$0^{11}$};
    \node at (1121) [left] {\tiny$10^{10}$};
    \node at (0121) [left] {\tiny$10^310^6$};
    \node at (1111) [left] {\tiny$1^20^9$};
    \node at (0111) [below=1pt, left=-1pt] {\tiny$1^20^210^6$};
    \node at (0011) [left=] {\tiny$1^20^21^20^5$};
    \node at (1110) [left] {\tiny$1^30^8$};
    \node at (0110) [below=1pt, left=-1pt] {\tiny$1^3010^6$};
    \node at (0010) [below=1pt, left=-1pt] {\tiny$1^301^20^5$};
    \node at (h3) [below] {\tiny$1^301^2010^3$};
    \node at (1100) [left] {\tiny$1^40^7$};
    \node at (1000) [left] {\tiny$1^4010^5$};
    \node at (h1) [above] {\tiny$1^401^20^4$};
    \node at (0100) [below=1pt, left=-1pt] {\tiny$1^50^6$};
    \node at (h2) [left=6pt, above=-1pt] {\tiny$1^60^5$};
    \node at (-1000) [right] {\tiny$1^70^4$};
    \node at (-0010) [above=1pt, right] {\tiny$1^6010^3$};
    \node at (-0011) [right] {\tiny$1^601^20^2$};
    \node at (-0100) [above=1pt, right] {\tiny$1^520^5$};
    \node at (-0110) [above=1pt, right] {\tiny$1^52010^3$};
    \node at (-0111) [above=1pt, right] {\tiny$1^5201^20^2$};
    \node at (-0121) [right] {\tiny$1^5201^30$};
    \node at (-1100) [right] {\tiny$1^5210^4$};
    \node at (-1110) [right] {\tiny$1^521^20^3$};
    \node at (-1111) [right] {\tiny$1^521^30^2$};
    \node at (-1121) [right] {\tiny$1^521^40$};
    \node at (-1221) [right] {\tiny$1^521^5$};
\end{tikzpicture}
\]
\[
\begin{tikzpicture}[scale=.75]
    \coordinate (1221) at (-6,-1);
    \coordinate (1121) at (-5,0);
    \coordinate (1111) at (-4,1);
    \coordinate (0121) at (-4,-1);
    \coordinate (1110) at (-3,2);
    \coordinate (0111) at (-3,0);
    \coordinate (1100) at (-2,3);
    \coordinate (0110) at (-2,1);
    \coordinate (0011) at (-2,-1);
    \coordinate (1000) at (-1,4);
    \coordinate (0100) at (-1,2);
    \coordinate (0010) at (-1,0);
    \coordinate (h1) at (0,4);
    \coordinate (h2) at (0,2);
    \coordinate (h3) at (0,0);
    \coordinate (-1000) at (1,4);
    \coordinate (-0100) at (1,2);
    \coordinate (-0010) at (1,0);
    \coordinate (-1100) at (2,3);
    \coordinate (-0110) at (2,1);
    \coordinate (-0011) at (2,-1);
    \coordinate (-1110) at (3,2);
    \coordinate (-0111) at (3,0);
    \coordinate (-1111) at (4,1);
    \coordinate (-0121) at (4,-1);
    \coordinate (-1121) at (5,0);
    \coordinate (-1221) at (6,-1);
    \draw
        (1221) -- (1000) -- (-1000) -- (-1221)
        (1121) -- (0121) -- (0100) -- (-0100) -- (-0121) -- (-1121)
        (1111) -- (0011) -- (0010) -- (-0010) -- (-0011) -- (-1111)
        (1110) -- (0010)    (-0010) -- (-1110)
        (1100) -- (0100)    (-0100) -- (-1100)
    ;
    \draw
        (h1) -- (0100)
        (h1) -- (-0100)
        (h2) -- (1000)
        (h2) -- (-1000)
        (h2) -- (0010)
        (h2) -- (-0010)
        (h3) -- (0100)
        (h3) -- (-0100)
    ;
    \draw[ultra thick, red]
        (1221) -- (1121) -- (0121) -- (0111) -- (0011) -- (0010) -- (-0010) -- (-0011) -- (-0111) -- (-0121) -- (-1121) -- (-1221)
        (1121) -- (1000) -- (h1)
        (0111) -- (0100) -- (h2)
        (-0100) -- (-1100)
        (-0010) -- (-1110)
        (-0111) -- (-1111)
    ;
    \draw[ultra thick, red]
        (h2) -- (-1000)
        (h3) -- (-0100)
    ;
    \draw[black, fill=black]
        (1221) circle (3pt)
        (1121) circle (3pt)
        (1111) circle (3pt)
        (0121) circle (3pt)
        (1110) circle (3pt)
        (0111) circle (3pt)
        (1100) circle (3pt)
        (0110) circle (3pt)
        (0011) circle (3pt)
        (1000) circle (3pt)
        (0100) circle (3pt)
        (0010) circle (3pt)
        (-1000) circle (3pt)
        (-0100) circle (3pt)
        (-0010) circle (3pt)
        (-1100) circle (3pt)
        (-0110) circle (3pt)
        (-0011) circle (3pt)
        (-1110) circle (3pt)
        (-0111) circle (3pt)
        (-1111) circle (3pt)
        (-0121) circle (3pt)
        (-1121) circle (3pt)
        (-1221) circle (3pt)
    ;
    \draw[black, fill=white]
        (h1) circle (3pt)
        (h2) circle (3pt)
        (h3) circle (3pt)
    ;
    \node[draw] at (-4,3) {$\nu_3$};
    \node at (1221) [left] {\tiny$0^{11}$};
    \node at (1121) [left] {\tiny$10^{10}$};
    \node at (0121) [left] {\tiny$1^20^9$};
    \node at (1111) [left] {\tiny$1010^8$};
    \node at (0111) [below=1pt, left=-1pt] {\tiny$1^30^8$};
    \node at (0011) [left=] {\tiny$1^40^7$};
    \node at (1110) [left] {\tiny$101010^6$};
    \node at (0110) [below=1pt, left=-1pt] {\tiny$1^3010^6$};
    \node at (0010) [below=1pt, left=-1pt] {\tiny$1^50^6$};
    \node at (h3) [below] {\tiny$1^60^5$};
    \node at (1100) [left] {\tiny$1010110^5$};
    \node at (1000) [left] {\tiny$10101^2010^3$};
    \node at (h1) [above] {\tiny$10101^201010$};
    \node at (0100) [below=1pt, left=-1pt] {\tiny$1^301^20^5$};
    \node at (h2) [above] {\tiny$1^301^2010^3$};
    \node at (-1000) [right] {\tiny$1^301^201010$};
    \node at (-0010) [above=1pt, right] {\tiny$1^520^5$};
    \node at (-0011) [right] {\tiny$1^5210^4$};
    \node at (-0100) [above=1pt, right] {\tiny$1^6010^3$};
    \node at (-0110) [above=1pt, right] {\tiny$1^52010^3$};
    \node at (-0111) [above=1pt, right] {\tiny$1^521^20^3$};
    \node at (-0121) [right] {\tiny$1^521^30^2$};
    \node at (-1100) [right] {\tiny$1^601010$};
    \node at (-1110) [right] {\tiny$1^5201010$};
    \node at (-1111) [right] {\tiny$1^521^2010$};
    \node at (-1121) [right] {\tiny$1^521^40$};
    \node at (-1221) [right] {\tiny$1^521^5$};
\end{tikzpicture}
\]
The first two choices, $\bs_1$ and $\bs_2$, yield polytopes $\cP_{\X,\fwt[2],\bs_i}$ with volume less than $132=\deg(\X)$ (the first has volume 84, the second has volume 105), so neither of these confirms our conjecture. On the other hand, the third choice, $\bs_3$, yields a polytope $\cP_{\X,\varpi,\bs_3}$ that has volume $132=\deg(\X)$ and furthermore its lattice points form a Hilbert basis, confirming Conjectures \ref{conj:nobody} and \ref{conj:khovanskii}.

\end{document}